\documentclass[11pt,reqno]{amsart}

\usepackage{amssymb,amscd,amsxtra}

\usepackage{hyperref}
\hypersetup{colorlinks=false}

\theoremstyle{plain}

\newtheorem{thm}{Theorem}[section]
\newtheorem{lem}[thm]{Lemma}
\newtheorem{cor}[thm]{Corollary}
\newtheorem{prop}[thm]{Proposition}

\theoremstyle{definition}

\newtheorem{defn}[thm]{Definition}

\newtheorem{exs}[thm]{Examples}

\theoremstyle{remark}

\newtheorem{rem}{Remark}

\newcommand{\N}{\mathbb{N}}

\newcommand{\R}{\mathbb{R}}

\newcommand{\FF}{\mathcal{F}}

\newcommand{\HH}{\mathcal{H}}
\newcommand{\JJ}{\mathcal{J}}
\newcommand{\LL}{\mathcal{L}}

\newcommand{\RR}{\mathcal{R}}

\newcommand{\fX}{\mathfrak{X}}

\newcommand{\bH}{\mathbf{H}}
\newcommand{\bV}{\mathbf{V}}

\newcommand{\sA}{\mathsf{A}}

\newcommand{\sT}{\mathsf{T}}

\newcommand{\supp}{\operatorname{supp}}
\newcommand{\codim}{\operatorname{codim}}

\newcommand{\id}{\operatorname{id}}

\newcommand{\rnabla}{\mathring{\nabla}}
\newcommand{\rP}{\mathring{P}}
\newcommand{\rR}{\mathring{R}}
\newcommand{\rT}{\mathring{T}}
\newcommand{\rexp}{\mathring{\exp}}
\newcommand{\rGamma}{\mathring{\Gamma}}
\newcommand{\cexp}{\check{\exp}}
\newcommand{\cO}{\check{O}}
\newcommand{\cV}{\check{V}}
\newcommand{\vol}{\operatorname{vol}}

\usepackage{color}
\definecolor{darkgreen}{cmyk}{1,0,1,.2}
\definecolor{m}{rgb}{1,0.1,1}

\newdimen\theight
\def\TeXref#1{%
             \leavevmode\vadjust{\setbox0=\hbox{{\tt
                     \quad\quad  {\small \textrm #1}}}%
             \theight=\ht0
             \advance\theight by \lineskip
             \kern -\theight \vbox to
             \theight{\rightline{\rlap{\box0}}%
             \vss}%
             }}%

\title{Riemannian foliations of bounded geometry}

\author[J.A. \'Alvarez L\'opez]{Jes\'us A. \'Alvarez L\'opez}
\address{Departamento de Xeometr\'{\i}a e Topolox\'{\i}a\\
         Facultade de Matem\'aticas\\
         Universidade de Santiago de Compostela\\
         15782 Santiago de Compostela\\ Spain}
\email{jesus.alvarez@usc.es}
\thanks{The first author is partially supported by the MICINN grants MTM2008-02640 and MTM2011-25656}

\author[Y.A. Kordyukov]{Yuri A. Kordyukov}
\address{Institute of Mathematics\\ Russian Academy of Sciences\\
112 Chernyshevsky street\\ 450008 Ufa\\ Russia}
\email{yurikor@matem.anrb.ru}
\thanks{The second author is partially supported by the RFBR grant 12-01-00519 and by the MICINN grants MTM2008-02640 and MTM2011-25656.}

\author[E. Leichtnam]{Eric Leichtnam}
\address{Institut de Math\'ematiques de Jussieu-PRG, CNRS,  Batiment Sophie Germain (bureau 740), Case 7012, 75205 Paris Cedex 13, France}
\email{ericleichtnam@math.jussieu.fr}

\date{\today}

\subjclass{57R30, (53C21, 53C20)}

\keywords{Riemannian foliation, bounded geometry, O'Neill tensors}

\begin{document}

\maketitle

\begin{abstract}
	Continuing the study of bounded geometry for Riemannian foliations, begun by Sanguiao, we introduce a chart-free definition of this concept. Our main theorem states that it is equivalent to a condition involving certain normal foliation charts. For this type of charts, it is also shown that the derivatives of the changes of coordinates are uniformly bounded, and there are nice partitions of unity. Applications to a trace formula for foliated flows will given in a forthcoming paper.
\end{abstract}


\section{Introduction}\label{s:intro}

Bounded geometry has played an important role in global analysis on non-compact manifolds. Recall that a Riemannian manifold $M$ is said to be of bounded geometry if it has positive injectivity radius and all covariant derivatives of its curvature tensor are uniformly bounded; in particular, $M$ is complete. It also has a chart characterization: $M$ is of bounded geometry if and only if there are normal coordinates at each point $p$, whose image is an Euclidean ball $B$ independent of $p$, and such that the corresponding Christoffel symbols $\Gamma^i_{jk}$, as a family of functions on $B$ parametrized by $i$, $j$, $k$ and $p$, lie in a bounded set of the Fr\'echet space $C^\infty(B)$. In this characterization, the Cristoffel symbols can be replaced by the metric coefficients $g_{ij}$. This was first proved by Eichhorn \cite{Eichhorn1991}. A different proof was indicated by Roe \cite{Roe1988I}, using Rauch comparison theorem and integrating the differential equations that relate the curvature and the Christoffel symbols or metric coefficients \cite[Appendix]{AtiyahBottPatodi1973}, \cite{AtiyahBottPatodi1975}. The details of this proof are rather involved and were given by Schick \cite{Schick1996,Schick2001}. 

In fact, Schick studied bounded geometry for manifolds with boundary. In this case, besides the curvature and injectivity radius of the manifold, the chart-free definition also involves the curvature, second fundamental form and injectivity radius of the boundary, and the ``normal radius'' of a geodesic collar of the boundary. Moreover, besides normal coordinates at interior points, its chart characterization also involves the coordinates given by a geodesic collar of the boundary. 

For geometric structures on open manifolds that involve a metric, one may try to give a chart-free definition of bounded geometry by requiring some condition analogous to a positive injectivity radius (well adapted to the structure), and uniform bounds of the covariant derivatives of the tensors that describe the structure (including the curvature). The corresponding chart characterization should involve some version of normal coordinates canonically associated to the geometric structure.

An example of such generalization is the case of bounded geometry for Riemannian foliations, which was introduced by Sanguiao \cite[Definition~2.7]{Sanguiao2008}. It was used to extend certain analysis on Riemannian foliations to the case of open manifolds. But Sanguiao's definition is of chart type, and a chart-free counterpart was missing. 

We give a chart-free definition of bounded geometry for Riemannian foliations in the following way. This class of foliations is characterized by being locally defined by Riemannian submersions for some metric on the ambient manifold (a bundle-like metric). These local Riemannian submersions are described by the O'Neill tensors \cite[Theorem~4]{ONeill1966}, which can be combined to define global O'Neill tensors. Thus our definition requires the existence of uniform bounds of all covariant derivatives of the curvature and O'Neill tensors, and positivity of certain leafwise and transverse versions of the injectivity radius. 

Let us give some examples of Riemannian foliations that have bounded geometry for some bundle-like metric. The lifts of Riemannian foliations on compact manifolds to any connected covering have bounded geometry. Given a connected Lie group $H$, a connected normal Lie subgroup $L\subset H$, and a discrete subgroup $\Gamma\subset H$, the induced foliation on the homogeneous space $\Gamma\backslash H$, whose leaves are the projections of the translates of $L$, has bounded geometry. Given any codimension one foliation almost without holonomy on a compact manifold, with a finite number of compact leaves with non-trivial holonomy group, its restriction to the complement of those leaves has bounded geometry.

Our main theorem (Theorem~\ref{t:bounded geometry}) establishes the equivalence of this condition to a foliation chart characterization, which uses the metric coefficients with respect to what we call normal foliation coordinates. These coordinates, centered at a point $p$, are defined by taking geodesics orthogonal to the leaves emanating from $p$, defining a local transversal, and then taking geodesics in the leaves emanating from each point of this local transversal.

The proof follows the arguments of Schick. An important ingredient in the proof is a certain version of Jacobi fields, called adapted Jacobi fields. They are infinitesimal variations of leafwise geodesics in the ambient manifold, where leafwise geodesics are geodesics of the leaves, considered as curves in the ambient manifold. The key result about adapted Jacobi fields (Proposition~\ref{p:|X|^2+|Y|^2}) states that they are bounded in terms of the initial condition, the parameter, and the curvature and O'Neill tensors.

Two more well known results about bounded geometry are generalized to our setting: it is proved that the changes of coordinates of these normal foliation charts have uniformly bounded derivatives, and there are partitions of unity with uniformly bounded derivatives subordinated to an appropriate covering by domains of normal foliation coordinates.

\section{Preliminaries}\label{s:prelim}

Let $\FF$ be a smooth foliation on a (possibly non-compact) manifold $M$, and set $n'=\codim\FF$, $n''=\dim\FF$ and $n=\dim M=n'+n''$. Let us recall and fix the following familiar terminology and notation. The leaf through some point $p$ is usually denoted by $L_p$. Let $T\FF\subset TM$ be the subbundle of vectors tangent to the leaves, which is called {\em tangent\/} or {\em vertical\/} bundle of $\FF$. Then $N\FF=TM/T\FF$ is the bundle of vectors normal to the leaves, called the {\em normal\/} bundle of $\FF$. This gives rise to the concept of {\em vertical\/}/{\em normal\/} vectors, vector fields and local frames. Let $\fX(M)$ denote the Lie algebra and $C^\infty(M)$-module of tangent vector fields on $M$. The vertical vector fields form a Lie subalgebra and $C^\infty(M)$-submodule $\fX(\FF)\subset\fX(M)$. The normal vector (or normal vector field) induced by the projection of any tangent vector (or tangent vector field) $X$ will be denoted by $\overline{X}$. For any smooth local transversal $\Sigma$ of $\FF$ through a point $p$, there is a canonical isomorphism $T_p\Sigma\cong N_p\FF$. In the sequel, the notation $\Sigma$, with possible subindices, will be used for a smooth local transversal or the local quotient of a foliation chart, or a disjoint union of such objects. 

The normal bundle $N\FF$ has a flat $T\FF$-partial connection $\nabla^\FF$ given by \label{nabla^FF on N FF} $\nabla^\FF_V\overline{X}=\overline{[V,X]}$ for $V\in\fX(\FF)$ and $X\in\fX(M)$. For each path $c$ from $p$ to $q$ in a leaf, the corresponding holonomy transformation $h_c$ is defined between smooth local transversals through $p$ and $q$, and its differential can be considered as an isomorphism $h_{c*}:N_p\FF\to N_q\FF$, called the {\em infinitesimal holonomy\/} of $c$. This $h_{c*}$ equals the $\nabla^\FF$-parallel transport along $c$. 

The normalizer $\fX(M,\FF)$ of $\fX(\FF)$ in $\fX(M)$ consists of the tangent vector fields whose flows are {\em foliated\/} in the sense that they map leaves to leaves; these vector fields are called {\em infinitesimal transformations\/} of $(M,\FF)$. The quotient Lie algebra $\overline{\fX}(M,\FF)=\fX(M,\FF)/\fX(\FF)$ can be identified with the space of the $\nabla^\FF$-parallel normal vector fields (those that are invariant by all infinitesimal holonomy transformations). The element of $\overline{\fX}(M,\FF)$ represented by any $X\in\fX(M,\FF)$ is also denoted by $\overline{X}$.

Let $\{U_a,\pi_a,h_{ab}\}$ be a defining cocycle of $\FF$; here, the sets $U_a$ are distinguished open sets (domains of foliation charts) that form a covering of $M$, the maps $\pi_a:U_a\to\Sigma_a$ are the submersions whose fibers are the plaques, which will be called {\em distinguished submersions\/}, and the maps $h_{ab}:\pi_a(U_{ab})\to\pi_b(U_{ab})$ are the elementary holonomy transformations determined by the condition $h_{ab}\pi_a=\pi_b$ on $U_{ab}=U_a\cap U_b$. Let $\HH$ be the representative of the holonomy pseudogroup on $\Sigma=\bigsqcup_a\Sigma_a$ generated by the maps $h_{ab}$. Let $\fX(\Sigma,\HH)\subset\fX(\Sigma)$ be the Lie subalgebra consisting of $\HH$-invariant vector fields. Then $\fX(M,\FF)$ consists of the vector fields $X\in\fX(M)$ that can be projected by the distinguished submersions $\pi_a$, defining an element of $\fX(\Sigma,\HH)$, also denoted by $\overline{X}$. This assignment induces an isomorphism $\overline{\fX}(M,\FF)\cong\fX(\Sigma,\HH)$, which may be considered as an identity.

It is said that $\FF$ is {\em Riemannian\/} when there is an $\HH$-invariant Riemannian metric on $\Sigma$. This is equivalent to the existence of a $\nabla^\FF$-parallel Riemannian structure on $N\FF$. It also means that there is a Riemannian metric on $M$, called {\em bundle-like\/}, which induces a metric on $\Sigma$ so that each $\pi_a:U_a\to\Sigma_a$ is a Riemannian submersion. 

From now on, $\FF$ will be assumed to be Riemannian, and $M$ equipped with a bundle-like metric $g=(\ ,\ )$. Let $\nabla$ denote its Levi-Civita connection, and $R$ its curvature. The vector subbundle $T\FF^\perp\subset TM$ is called {\em horizontal\/}, giving rise to the concepts of {\em horizontal\/} vectors, vector fields and local frames. Let $\bV:TM\to T\FF$ and $\bH:TM\to T\FF^\perp$ denote the orthogonal projections. The Levi-Civita connection of the leaves can be combined to define a $T\FF$-partial connection of $T\FF$, also denoted by \label{nabla^FF on T FF} $\nabla^\FF$. The $\HH$-invariant metric induced by $g$ on $\Sigma$ will be denoted by $\check g$, and the corresponding Levi-Civita connection and exponential map by $\check\nabla$ and $\check\exp$.

\section{O'Neill tensors and adapted connection}\label{ss:O'Neill}

Only in this section, we will use the notation $V$ and $W$ for arbitrary vertical vector fields, $X$ and $Y$ for arbitrary horizontal vector fields, and $E$, $F$, $G$ and $H$ for arbitrary vector fields. The O'Neill tensors \cite{ONeill1966} of the Riemannian submersions $\pi_a$ combine to produce $(1,2)$-tensors $\sT$ and $\sA$ on $M$; they are defined by
  \begin{align*}
    \sT_EF&=\bH\nabla_{\bV E}(\bV F)+\bV\nabla_{\bV E}(\bH F)\;,\\
    \sA_EF&=\bH\nabla_{\bH E}(\bV F)+\bV\nabla_{\bH E}(\bH F)\;.
  \end{align*}
On the other hand, the expression
  \[
    \rnabla_EF=\bV\nabla_E(\bV F)+\bH\nabla_E(\bH F)
  \]
defines a linear connection $\rnabla$ \cite{AlvTond1991}, which is said to be {\em adapted\/}. Observe that
  \begin{equation}\label{nabla-rnabla}
    \nabla_V-\rnabla_V=\sT_V\;,\quad\nabla_X-\rnabla_X=\sA_X\;.
  \end{equation}
It is easy to check that $\rnabla$ is Riemannian; thus $\rnabla=\nabla$ if and only if $\rnabla$ is torsion-free.  In fact, the torsion $\rT$ of $\rnabla$ is given by
  	\begin{gather}
		\rT(V,W)=0\;,\quad\rT(X,Y)=-2\sA_XY\;,\label{rT(V,W)}\\
    		\rT(X,V)=-\rT(V,X)=\sT_VX-\sA_XV\;.\label{rT(X,V)}
  	\end{gather}
By using~\eqref{nabla-rnabla}, and since $\nabla$ is torsion free and $\bV[X,Y]=2A_XY$ \cite[Lemma~2]{ONeill1966}, it easily follows that the curvature tensor $\rR$ of $\rnabla$ satisfies
  	\begin{align}
      		(\rR-R)(V,W)&=-(\nabla_V\sT)_W+(\nabla_W\sT)_V+[\sT_V,\sT_W]\;,\label{rR(V,W)}\\
      		(\rR-R)(X,Y)&=-(\nabla_X\sA)_Y+(\nabla_Y\sA)_X+[\sA_X,\sA_Y]
		+\sT_{\bV[X,Y]}\;,\label{rR(X,Y)}\\
      		(\rR-R)(X,V)&=-(\nabla_X\sT)_V+(\nabla_V\sA)_X-\sT_{\sT_VX}+\sA_{\sA_XV}
		+[\sA_X,\sT_V]\;.\label{rR(X,V)}
  	\end{align}
$\rR(E,F)$ preserves the subspace of vertical/horizontal vector fields because $\rnabla_E$ preserves it. Since $\rnabla$ is Riemannian, the usual arguments \cite[Chapter~4, Proposition~2.5-(c)]{doCarmo1992} show that
	\begin{equation}\label{(rR(E,F)G,H)}
		(\rR(E,F)G,H)=-(\rR(E,F)H,G)\;.
	\end{equation}
Note that
  \begin{align}
    \nabla^\FF_VW&=\rnabla_VW\;,\label{rnabla_V W}\\
    \nabla^\FF_V\overline{X}&=\overline{\rnabla_VX-\sA_XV}\;,\label{nabla^FF_V overline X}
  \end{align}
where $\nabla^\FF$ denotes the $T\FF$-partial connection on $T\FF$ and $N\FF$ defined in Section~\ref{s:prelim}. Equation~\eqref{nabla^FF_V overline X} follows by taking local references of $T\FF^\perp$ consisting of local infinitesimal transformations, and using that $\rnabla_VX=\sA_XV$ if $X$ is a horizontal infinitesimal transformation by \cite[Lemma~3]{ONeill1966}.

By~\eqref{rnabla_V W}, the $\rnabla$-geodesics that are tangent to the leaves at some point remain tangent to the leaves at every point, and they are the geodesics of the leaves. So the leaves are $\rnabla$-totally geodesic, but not necessarily $\nabla$-totally geodesic, of course. By the second equality of~\eqref{nabla-rnabla} and \cite[Lemma~2]{ONeill1966}, $\rnabla$ and $\nabla$ have the same geodesics orthogonal to the leaves.

The connection $\breve\nabla=\rnabla-\frac{1}{2}\rT$ has the same geodesics as $\rnabla$, and is torsion-free but not Riemannian. We have
  \begin{gather*}
  	\breve\nabla_VW=\nabla_VW-\sT_VW\;,\quad
	\breve\nabla_XY=\nabla_XY\;,\\
		\begin{aligned}
			\breve\nabla_XV=\nabla_XV-\frac{1}{2}(\sT_VX+\sA_XV)\;,\\
			\breve\nabla_VX=\nabla_VX-\frac{1}{2}(\sT_VX+\sA_XV)\;.
		\end{aligned}
 \end{gather*}
As usual, $\rnabla$ and $\breve\nabla$ induce connections on all tensor bundles over $M$.

Observe that the expression $\theta_XV=\bV[X,V]$ defines a differential operator $\theta:\fX(\FF)\to C^\infty(M;T\FF^{\perp\,*}\otimes T\FF)$; indeed, by using that $\nabla$ is torsion free, it follows that
  \begin{equation}\label{theta_XV}
    \theta_XV=\rnabla_XV-\sT_VX\;.
  \end{equation}

\section{Adapted exponential map}\label{s: rexp}

Let $\rexp$ be the exponential map of the geodesic spray of $\rnabla$ \cite[pp.~96--99]{Poor1981}, which can be called {\em adapted\/} exponential map. If $\rnabla$ is complete, then $\rexp$ is defined on the whole of $TM$. Observe that the exponential map of the leaves is given by the restriction of $\rexp$ to $T\FF$. 

For each $p\in M$, $\rexp$ defines a diffeomorphism of some open neighborhood $V$ of $0$ in $T_pM$ to some open neighborhood $O$ of $p$ in $M$. Suppose that $p\in U_a$. Then, by taking $V$ small enough, we can suppose that $O\subset U_a$. The map $\cexp$ restricts to a diffeomorphism of some open neighborhood $\check V$ of $0$ in $T_{\pi_a(p)}\Sigma$ to some neighborhood $\check O$ of $\pi_a(p)$ in $\Sigma$. If $V$ is again small enough, then we can assume that $\pi_{a*}(V)\subset\check V$ and $\pi_a(O)\subset\check O$. From \cite[Lemma~1-(3)]{ONeill1966}, it follows that, on $V$,
	\begin{equation}\label{pi rexp = cexp pi_*}
		\pi_a\,\rexp=\cexp\,\pi_{a*}\;.
	\end{equation} 

\begin{prop}\label{p:rexp}
	For $E,F\in V$, we have $E-F\in T_p\FF$ if and only if $\rexp E$ and $\rexp F$ are in the same plaque.
\end{prop}

\begin{proof}
	Since $\cexp$ is injective on $\check V$ and by~\eqref{pi rexp = cexp pi_*}, we get
		\begin{multline*}
			E-F\in T_p\FF\Longleftrightarrow\pi_{a*}E=\pi_{a*}F\Longleftrightarrow\cexp\,\pi_{a*}E
			=\cexp\,\pi_{a*}F\\
			\Longleftrightarrow\pi_a\,\rexp E=\pi_a\,\rexp F\;.\qed
		\end{multline*}
\renewcommand{\qed}{}
\end{proof}

\section{Adapted Jacobi fields}\label{s:Jacobi}

Let $\gamma:[a,b]\to M$ ($a\le b$) be a fixed $\rnabla$-geodesic in some leaf of $\FF$, which will be called a {\em leafwise geodesic\/}, and let $\fX(M,\gamma)$ denote the linear space of tangent vector fields of $M$ along $\gamma$. Some $X\in\fX(M,\gamma)$ is called an {\em adapted Jacobi field\/} if
	\begin{gather}
		\rnabla_{\dot\gamma}\rnabla_{\dot\gamma}X-\rR(\dot\gamma,X)\dot\gamma-\rnabla_{\dot\gamma}(\rT(\dot\gamma,X))=0\;,\label{Jacobi}\\
		\rnabla_{\dot\gamma(a)}X-\rT(\dot\gamma(a),X(a))\in T_{\gamma(a)}\FF\;.\label{vertical}
	\end{gather}
This type of vector fields forms a linear subspace $\JJ(M,\FF,\gamma)\subset\fX(M,\gamma)$. 

\begin{prop}\label{p:vertical}
	$\rnabla_{\dot\gamma}X-\rT(\dot\gamma,X)$ is vertical for all $X\in\JJ(M,\FF,\gamma)$.
\end{prop}

\begin{proof}
	With $Y=\rnabla_{\dot\gamma}X-\rT(\dot\gamma,X)$, the conditions~\eqref{Jacobi} and~\eqref{vertical} become
		\begin{equation}\label{Y}
			\rnabla_{\dot\gamma}Y-\rR(\dot\gamma,X)\dot\gamma=0\;,\quad Y(a)\in T_{\gamma(a)}\FF\;.
		\end{equation}
	Let $\{e_1,\dots,e_n\}$ be a $\rnabla$-parallel orthonormal frame of $TM$ along $\gamma$ so that $e_i$ is vertical if $i>n'$. Write\footnote{In products, we use the convention that repeated indices are summed.} $Y=Y^ie_i$ with $Y^i\in C^\infty([a,b])$. Since $\rR(\dot\gamma,X)\dot\gamma$ is vertical, it follows from~\eqref{Y} that $\dot Y^i=0$ and $Y^i(a)=0$ if $i\le n'$. So $Y^i=0$ for $i\le n'$, obtaining that $Y$ is vertical.
\end{proof}

\begin{cor}\label{c:nabla^FF_dot gamma overline X=0}
	$\nabla^\FF_{\dot\gamma}\overline X=0$ for all $X\in\JJ(M,\FF,\gamma)$.
\end{cor}

\begin{proof}
	By~\eqref{rT(V,W)} and~\eqref{rT(X,V)},
		\begin{align*}
			\rnabla_{\dot\gamma}X-\rT(\dot\gamma,X)
			&=\rnabla_{\dot\gamma}\bV X+\rnabla_{\dot\gamma}\bH X-\rT(\dot\gamma,\bV X)-\rT(\dot\gamma,\bH X)\\
			&=\rnabla_{\dot\gamma}\bV X+\rnabla_{\dot\gamma}\bH X
			-\sA_{\bH X}\,\dot\gamma+\sT_{\dot\gamma}\bH X\;,
		\end{align*}
	where $\rnabla_{\dot\gamma}\bV X$ and $\sT_{\dot\gamma}\bH X$ are vertical, and $\rnabla_{\dot\gamma}\bH X$ and $\sA_{\bH X}\,\dot\gamma$ are horizontal. So $\rnabla_{\dot\gamma}\bH X-\sA_{\bH X}\,\dot\gamma=0$ by Proposition~\ref{p:vertical}, giving $\nabla^\FF_{\dot\gamma}\overline X=\nabla^\FF_{\dot\gamma}\overline{\bH X}=0$ by~\eqref{nabla^FF_V overline X}.
\end{proof}

\begin{cor}\label{c:|sH X| is constant}
	The function $|\bH X|$ is constant for all $X\in\JJ(M,\FF,\gamma)$.
\end{cor}

\begin{proof}
	This follows from Corollary~\ref{c:nabla^FF_dot gamma overline X=0} since the metric induced by $g$ on $N\FF$ is $\nabla^\FF$-parallel.
\end{proof}

\begin{prop}\label{p:existence of Jacobi}
	For all $X_0\in T_{\gamma(a)}M$ and $Y_0\in T_{\gamma(a)}\FF$, there is a unique $X\in\JJ(M,\FF,\gamma)$ such that 
		\begin{equation}\label{initial conditions}
			X(a)=X_0\;,\quad\rnabla_{\dot\gamma(a)}X-\rT(\dot\gamma(a),X_0)=Y_0\;.
		\end{equation}
\end{prop}

\begin{proof}
	Without loss of generality, we can assume that $\gamma$ is parametrized by arc length. Then we can choose $\{e_1,\dots,e_n\}$, like in the proof of Proposition~\ref{p:vertical}, so that $e_n=\dot\gamma$. Write $\rR(e_k,e_l)e_j=\rR_{jkl}^i\,e_i$, $\rT(e_k,e_l)=\rT_{kl}^i\,e_i$, $X_0=X_0^i\,e_i(a)$ and $Y_0=Y_0^i\,e_i(a)$ with $\rR_{jkl}^i,\rT_{kl}^i\in C^\infty([a,b])$ and $X_0^i,Y_0^i\in\R$. For $X=X^ie_i\in\fX(M,\gamma)$ ($X^i\in C^\infty([a,b])$),~\eqref{Jacobi} and~\eqref{initial conditions} become
		\begin{gather*}
			\ddot X^i-\rT_{nj}^i\,\dot X^j-(\rR_{nnj}^i+\rT_{nj}^i)X^j=0\;,\\
			X^i(a)=X_0^i\;,\quad\dot X^i(a)=Y_0^i+X_0^j\,\rT_{nj}^i(a)\;.
		\end{gather*}
	By the basic ODE theory, this system of differential equations has a unique solution $(X^1,\dots,X^n)$. So there is a unique $X\in\fX(M,\gamma)$ satisfying~\eqref{Jacobi} and~\eqref{initial conditions}. Moreover $X$ satisfies~\eqref{vertical} because $Y_0\in T_{\gamma(a)}\FF$.
\end{proof}

Proposition~\ref{p:existence of Jacobi} means that the mapping
	\[
		X\mapsto(X(a),\rnabla_{\dot\gamma(a)}X-\rT(\dot\gamma(a),X(a)))
	\]
defines a linear isomorphism $\JJ(M,\FF,\gamma)\to T_{\gamma(a)}M\oplus T_{\gamma(a)}\FF$, and therefore $\dim\JJ(M,\FF,\gamma)=n+n''$.

A {\em leafwise geodesic variation\/} of $\gamma$ is a smooth map $f:[a,b]\times(-\epsilon,\epsilon)\to M$, for some $\epsilon>0$, such that each curve $\gamma_s=f(\cdot,s)$ is a leafwise geodesic, with $\gamma_0=\gamma$. The {\em variation field\/} of $f$, $X\in\fX(M,\gamma)$, is defined by
	\[
		X(t)=\partial_sf(t,s)|_{s=0}\;.
	\]

\begin{prop}\label{p:Jacobi=variation}
	$\JJ(M,\FF,\gamma)$ consists of the variation fields of leafwise geodesic variations of $\gamma$.
\end{prop}

\begin{proof}
	Suppose that $X\in\fX(M,\gamma)$ is the variation field of a leafwise geodesic variation $f:[a,b]\times(-\epsilon,\epsilon)\to M$ of $\gamma$. Let $(t,s)$ denote the canonical coordinates of $[a,b]\times(-\epsilon,\epsilon)$. Observe that
		\begin{equation}\label{rnabla_partial_tf partial_sf - rT(partial_tf,partial_sf)}
			\rnabla_{\partial_tf}\partial_sf-\rT(\partial_tf,\partial_sf)=\rnabla_{\partial_sf}\partial_tf
		\end{equation}
	is vertical because $\partial_tf$ is vertical and $\rnabla_{\partial_sf}$ preserves vertical vector fields. Then, by evaluating this equality at $(s,t)=(0,0)$, we get that $X$ satisfies~\eqref{vertical}. Moreover, since $\rnabla_{\partial_tf}\partial_tf=[\partial_tf,\partial_sf]=0$, we get
		\begin{align*}
			\rnabla_{\partial_tf}\rnabla_{\partial_tf}\partial_sf
			&=\rnabla_{\partial_tf}\rnabla_{\partial_sf}\partial_tf+\rnabla_{\partial_tf}(\rT(\partial_tf,\partial_sf))\\
			&=\rR(\partial_tf,\partial_sf)\partial_tf+\rnabla_{\partial_tf}(\rT(\partial_tf,\partial_sf))\;.
		\end{align*}
	By evaluating this equality at $s=0$, it follows that $X$ satisfies~\eqref{Jacobi}. Therefore $X\in\JJ(M,\FF,\gamma)$.
	
	Now take any $X\in\JJ(M,\FF,\gamma)$, and let
		\[
			X_0=X(a)\;,\quad Y_0=\rnabla_{\dot\gamma(a)}X-\rT(\dot\gamma(a),X_0)\;.
		\]
	Let $\xi:(-\epsilon,\epsilon)\to M$ ($\epsilon>0$) be a $\rnabla$-geodesic with initial conditions $\xi(0)=\gamma(a)$ and $\dot\xi(0)=X_0$. Let $V$ and $W$ be the $\rnabla$-parallel tangent vector fields along $\xi$ with $V(0)=\dot\gamma(a)$ and $W(0)=Y_0$. Note that $V$ and $W$ are vertical because $\dot\gamma(a)$ and $Y_0$ are vertical, and the $\rnabla$-parallel transport preserves vertical vector fields. Let $f:[a,b]\times(-\epsilon,\epsilon)\to M$ be the leafwise geodesic variation of $\gamma$ given by 
		\[
			f(t,s)=\rexp_{\xi(s)}((t-a)(V(s)+sW(s)))\;.
		\]
	 The corresponding variation field $Z\in\JJ(M,\FF,\gamma)$ satisfies $Z(a)=\partial_sf(a,0)=\dot\xi(0)=X_0$ and
		\begin{multline*}
			\rnabla_{\dot\gamma(a)}Z-\rT(\dot\gamma(a),Z_0)
			=\rnabla_{\partial_tf}\partial_sf|_{s=0}-\rT(\partial_tf,\partial_sf)|_{s=0}
			=\rnabla_{\partial_sf}\partial_tf|_{s=0}\\
			=\rnabla_{X_0}(V(s)+sW(s))=W(0)
			=Y_0=\rnabla_{\dot\gamma(a)}X-\rT(\dot\gamma(a),X_0)\;.
		\end{multline*}
	So $Z=X$ by Proposition~\ref{p:existence of Jacobi}.	
\end{proof}

\begin{rem}
	Proposition~\ref{p:vertical} also follows from Proposition~\ref{p:Jacobi=variation} and evaluating the terms of~\eqref{rnabla_partial_tf partial_sf - rT(partial_tf,partial_sf)} at $s=0$.
\end{rem}

\begin{prop}\label{p:|X|^2+|Y|^2}
	Let $X\in\JJ(M,\FF,\gamma)$ and $Y=\rnabla_{\dot\gamma}X-\rT(\dot\gamma,X)$. Then 
		\[
			|X(b)|^2+|Y(b)|^2\le e^{C(b-a)}\left(|X(a)|^2+|Y(a)|^2\right)\;,
		\]
	where\footnote{Recall that  $\gamma:[a,b]\to M$ is a leafwise geodesic, and therefore $|\dot\gamma|$ is independent of $t$.}
		\[
			C=\max_{a\le t\le b}\max\left\{3,2|\rT_{\gamma(t)}|^2|\dot\gamma|^2+1+|\rR_{\gamma(t)}|^2|\dot\gamma|^4\right\}\;.
		\]
\end{prop}

\begin{proof}
	We have
		\begin{multline*}
			|\rnabla_{\dot\gamma}X|^2\le(|Y|+|\rT(\dot\gamma,X)|)^2
			=|Y|^2+|\rT(\dot\gamma,X)|^2+2|Y||\rT(\dot\gamma,X)|\\
			\le2(|Y|^2+|\rT(\dot\gamma,X)|^2)
			\le2(|Y|^2+|\rT_\gamma|^2|\dot\gamma|^2|X|^2)\;.
		\end{multline*}
	So
		\begin{multline*}
			\frac{d}{dt}|X|^2=2(\rnabla_{\dot\gamma}X,X)\le2|\rnabla_{\dot\gamma}X||X|
			\le|\rnabla_{\dot\gamma}X|^2+|X|^2\\
			\le2|Y|^2+(2|\rT_\gamma|^2|\dot\gamma|^2+1)|X|^2\;.
		\end{multline*}
	Moreover, by~\eqref{Jacobi},
		\begin{multline*}
			\frac{d}{dt}|Y|^2=2(\rnabla_{\dot\gamma}Y,Y)=2(\rR(\dot\gamma,X)\dot\gamma,Y)
			\le2|\rR(\dot\gamma,X)\dot\gamma||Y|\\
			\le|\rR(\dot\gamma,X)\dot\gamma|^2+|Y|^2
			\le|\rR_\gamma|^2|\dot\gamma|^4|X|^2+|Y|^2\;.
		\end{multline*}
	Thus
		\[
			\frac{d}{dt}\left(|X|^2+|Y|^2\right)\le C\left(|X|^2+|Y|^2\right)
		\]
	for the constant $C$ defined in the statement, and the result follows.
\end{proof}

Like in the study of usual Jacobi fields, $\dot\gamma$ and $(t-a)\dot\gamma$ are adapted Jacobi fields along $\gamma$. However it may not be possible to modify some $X\in\JJ(M,\FF,\gamma)$ so that it becomes an adapted Jacobi field orthogonal to $\dot\gamma$. But we can modify it as follows so that the corresponding modification of $Y=\rnabla_{\dot\gamma}X-\rT(\dot\gamma,X)$ is orthogonal to $\dot\gamma$.  Since $\rnabla$ is metric, and by~\eqref{Y} and~\eqref{(rR(E,F)G,H)},
	\[
		\frac{d}{dt}(Y,\dot\gamma)=(\rnabla_{\dot\gamma}Y,\dot\gamma)=(\rR(\dot\gamma,X)\dot\gamma,\dot\gamma)=0\;.
	\]
Hence $(Y,\dot\gamma)$ is constant, say equal to some $c\in\R$, and $Y^*=Y-c\dot\gamma$ is vertical and orthogonal to $\dot\gamma$. Moreover $X^*=X-c(t-a)\dot\gamma\in\JJ(M,\FF,\gamma)$ satisfies
	\[
		\rnabla_{\dot\gamma}X^*-\rT(\dot\gamma,X^*)=\rnabla_{\dot\gamma}X-\rT(\dot\gamma,X)-c\dot\gamma=Y^*\;.
	\]

\section{Normal foliation coordinates}\label{s: normal foln coordinates}

For a fixed point $p\in M$, let $\kappa_p$ (or simply $\kappa$) be the smooth map of some neighborhood $W$ of $0$ in $T_pM$ to $M$ defined by
	\[
		\kappa_p(X)=\rexp_q(\rP_{\bH X}\bV X)\;,
	\]
where $q=\rexp_p(\bH X)$, and $\rP_{\bH X}:T_pM\to T_qM$ denotes the $\rnabla$-parallel transport along the $\rnabla$-geodesic $t\mapsto\rexp_p(t\bH X)$, $0\le t\le1$. Take $\rexp:V\to O$, $\pi_a:U_a\to\Sigma_a$ and $\cexp:\cV\to\cO$ like in Section~\ref{s: rexp}. By choosing $W$ small enough, we have $W\subset V$ and $\kappa(W)\subset O$; thus $\pi_{a*}(W)\subset\cV$ and $\pi_a\kappa(W)\subset\cO$. For $X\in V$, by the definition of $\kappa$ and~\eqref{pi rexp = cexp pi_*}, we have 
	\begin{equation}\label{pi kappa = cexp pi_*}
		\pi_a\kappa(X)=\pi_a\kappa(\bH X)=\pi_a\,\rexp(\bH X)=\cexp\,\pi_{a*}\bH X=\cexp\,\pi_{a*}X\;.
	\end{equation}
Like in Proposition~\ref{p:rexp}, we get the following from~\eqref{pi kappa = cexp pi_*}.

\begin{prop}\label{p:kappa is foliated}
	For $X,Y\in W$, we have $X-Y\in T_p\FF$ if and only if $\kappa(X)$ and $\kappa(Y)$ belong to the same plaque of $U$.
\end{prop}

\begin{prop}\label{p:kappa_*=id}
	$\kappa_*\equiv\id:T_0(T_pM)\equiv T_pM\to T_pM$.
\end{prop}

\begin{proof}
	We have $\kappa_*\equiv\id:T_0(T_p\FF)\equiv T_p\FF\to T_p\FF$ because, on $T_p\FF$, $\kappa$ equals the exponential map of the leaf $L_p$. Moreover $\kappa_*(T_0(T_p\FF^\perp))\subset T_p\FF^\perp$ by the definition of $\kappa$, and~\eqref{pi kappa = cexp pi_*} induces the commutative diagram
		\[
			\begin{CD}
				T_0(T_p\FF^\perp) @>{\kappa_*}>> T_p\FF^\perp\\
				@V{\pi_{a**}}VV  @VV{\pi_{a*}}V \\
				T_0(T_{\pi_a(p)}\Sigma) @>{\cexp_*}>> T_{\pi_a(p)}\Sigma\;,
			\end{CD}
		\]
	where the vertical maps are isomorphisms. Hence $\kappa_*\equiv\id:T_0(T_p\FF^\perp)\equiv T_p\FF^\perp\to T_p\FF^\perp$.
\end{proof}

\begin{cor}\label{c:kappa is a chart}
	$\kappa$ defines a diffeomorphism of some neighborhood of $0$ in $T_pM$ to some neighborhood of $p$ in $M$.
\end{cor}

By choosing orthonormal references, we get identities $T_p\FF^\perp\equiv\R^{n'}$ and $T_p\FF\equiv\R^{n''}$. Then, according to Proposition~\ref{p:kappa is foliated} and Corollary~\ref{c:kappa is a chart}, for some open balls centered at the origin, $B'$ in $\R^{n'}$ and $B''$ in $\R^{n''}$, we can assume that $\kappa$ is a diffeomorphism of $B'\times B''$ to some open neighborhood $U$ of $p$, and $\kappa^{-1}=x=(x^1,\dots,x^n)$ is a foliation chart on $U$; it will be said that $x$ is a system of {\em normal\/} foliation coordinates or {\em normal\/} foliation chart at $p$. Write $x\equiv(x',x'')$ with $x'=(x^1,\dots,x^{n'})$ and $x''=(x^{n'+1},\dots,x^n)$. We will use the identity $U\equiv B'\times B''$ given by these coordinates; in particular,  the local transversal defined by $x''=0$ will be identified with $B'\times\{0\}$. As usual, the notation $g_{ij}$ is used for the corresponding coefficients of the metric $g$, and let $(g^{ij})=(g_{ij})^{-1}$. The more explicit notation $U_p$, $x_p$, $x_p'$,  $x_p''$, $x_p^i$, $x_p^{\prime i}$,  $x_p^{\prime\prime i}$, $g^p_{ij}$ and $g_p^{ij}$ may be also used.

\begin{rem}
	The exponential map $\rexp$ also gives rise to foliation coordinates according to Proposition~\ref{p:rexp}. The foliation coordinates defined by $\kappa$ are used to avoid having to study a more general version of adapted Jacobi fields (along arbitrary $\rnabla$-geodesics); specially, to avoid having to estimate their norm, which would possibly need some adapted version of the Rauch comparison theorem. With $\kappa$, the estimates given by Corollary~\ref{c:|sH X| is constant} and Proposition~\ref{p:|X|^2+|Y|^2} will be enough for our purposes. However, the study of more general adapted Jacobi fields could have its own interest.
\end{rem}

\section{Coefficients with respect to normal foliation coordinates}\label{s: coefficients}

Consider the notation of Section~\ref{s: normal foln coordinates}. The reference $\partial_1(0),\dots,\partial_n(0)$ of $T_pM$ is orthonormal. Take the $\rnabla$-parallel transport of $\partial_1(0),\dots,\partial_n(0)$ along the geodesics orthogonal to the leaves emanating from $p$, defining an orthonormal frame on $B'\times\{0\}$. Then extend this orthonormal frame on $B'\times\{0\}$ to $U\equiv B'\times B''$ by using the $\rnabla$-parallel transport along leafwise geodesics emanating from $B'\times\{0\}$. This process produces an orthonormal frame $s_1,\dots,s_n$ of $TM|_U$. Let $\theta^1,\dots,\theta^n$ be the orthonormal reference of $TM^*|_U$ dual to $s_1,\dots,s_n$. Write $\theta^i=a^i_j\,dx^j$ and $dx^i=b^i_j\,\theta^j$; thus $(a^i_j)=(b^i_j)^{-1}$, $\partial_j=a^i_js_i$ and $s_j=b^i_j\partial_i$. The orthonormality of $\theta^1,\dots,\theta^n$ means $g_{ij}\,dx^i\otimes dx^j=\theta^\alpha\otimes\theta^\alpha$, and therefore
	\begin{equation}\label{g_ij}
		g_{ij}=a^\alpha_ia^\alpha_j\;,\quad g^{ij}=b_\alpha^ib_\alpha^j\;.
	\end{equation}

To emphasize the difference between transverse and leafwise coordinates, we may use the notation $x^{\prime i}=x^i$, $\partial'_i=\partial_i$ and $\theta^{\prime i}=\theta^i$ for $i\le n'$, and $x^{\prime\prime i}=x^i$, $\partial''_i=\partial_i$ and $\theta^{\prime\prime i}=\theta^i$ for $i>n'$. Thus, when using $x^{\prime i}$, $\partial'_i$ or $\theta^{\prime i}$, it will be understood that $i$ runs in $\{1,\dots,n'\}$, and, when using $x^{\prime\prime i}$, $\partial''_i$ or $\theta^{\prime\prime i}$, it will be understood that $i$ runs in $\{n'+1,\dots,n\}$. Observe that, on $U$,
	\begin{equation}\label{sV, sH}
		\bV=g_{ik}g^{kj}\,\partial''_j\otimes dx^{\prime i}+\partial''_i\otimes dx^{\prime\prime i}\;,\quad
		\bH=\partial'_i\otimes dx^{\prime i}-g_{ik}g^{kj}\,\partial''_j\otimes dx^{\prime i}\;,
	\end{equation}
where $k$ runs in $\{n'+1,\dots,n\}$.

Given any $x=(x',x'')\in B'\times B''$, $t\mapsto(x',tx'')$ is identified with a leafwise geodesic via the normal foliation chart.

\begin{prop}\label{partial'_i(x',tx'') is Jacobi}
	$t\mapsto\partial'_i(x',tx'')$ is an adapted Jacobi field along the leafwise geodesic $t\mapsto(x',tx'')$.
\end{prop}

\begin{proof}
	Let $e'_i$ denote the $i$th element of the canonical reference of $\R^{n'}$. The map $(t,s)\mapsto(x'+se'_i,tx'')$ is a leafwise geodesic variation of $t\mapsto(x',tx'')$ whose variation field is $t\mapsto\partial'_i(x',tx'')$. Hence the result follows from Proposition~\ref{p:Jacobi=variation}.
\end{proof}

The standard index notation is used for the coefficients of any tensor with respect to the references $\partial_1,\dots,\partial_n$ and $dx^1,\dots,dx^n$, and the indices will be underlined in the coefficients with respect to the frames $s_1,\dots,s_n$ and $\theta^1,\dots,\theta^n$. We may even mix non-underlined and underlined indices. For instance, $\rT=\rT^{\underline i}\,s_i$ and $\rR=\rR^{\underline i}_{\underline j}\,s_i\otimes\theta^j$ for $2$-forms 
	\[
		\rT^{\underline i}=\frac{1}{2}\rT^{\underline i}_{kl}\,dx^k\wedge dx^l\;,\quad
		\rR^{\underline i}_{\underline j}=\frac{1}{2}\rR^{\underline i}_{\underline{j}kl}\,dx^k\wedge dx^l\;,
	\]
where $\rT^{\underline i}_{kl}=-\rT^{\underline i}_{lk}$ and $\rR^{\underline i}_{\underline{j}kl}=-\rR^{\underline i}_{\underline{j}lk}$; note that $\rT^{\underline i}_{jk}=\rT^{\underline i}(\partial_j,\partial_k)$ and $\rR^{\underline i}_{\underline{j}kl}=\rR^{\underline i}_{\underline{j}}(\partial_j,\partial_k)$. This notation is also used for the coefficients of the covariant derivatives; for example, $\rnabla_{\partial_m}\rT=\rT^{\underline i}_{;m}s_i$ and $\rnabla_{\partial_m}\rT^{\underline i}=(\rT^{\underline i})_{;m}s_i$ for $2$-forms 
	\[
		\rT^{\underline i}_{;m}=\frac{1}{2}\rT^{\underline i}_{kl;m}\,dx^k\wedge dx^l\;,\quad
		(\rT^{\underline i})_{;m}=\frac{1}{2}(\rT^{\underline i})_{kl;m}\,dx^k\wedge dx^l\;,
	\]
where $\rT^{\underline i}_{kl;m}=-\rT^{\underline i}_{lk;m}$ and $(\rT^{\underline i})_{kl;m}=-(\rT^{\underline i})_{lk;m}$. A similar notation is used for the $\rnabla$-Christoffel symbols: $\rnabla_{\partial_k}\partial_j=\rGamma^i_{jk}\partial_i$ and $\rnabla_{\partial_k}s_j=\rGamma^{\underline i}_{\underline{j}k}s_i$, and therefore $\rnabla_{\partial_k}dx^j=-\rGamma^j_{ik}dx^i$ and $\rnabla_{\partial_k}\theta^j=-\rGamma^{\underline{j}}_{\underline{i}k}\theta^i$. By the definition of $s_1,\dots,s_n$, we have
	\begin{gather}
		\rGamma^{\underline i}_{\underline{j}k}(0)=0\;,
		\label{rGamma^underline i_underline j k(0)=0}\\
		\rGamma^{\underline i}_{\underline{j}k}(x',0)=0\quad\forall x'\in B'\quad\text{if}\quad k>n'\;.
		\label{rGamma^underline i_underline j k(x',0)=0}
	\end{gather}
Furthermore
	\begin{equation}\label{rT^underline i_;m}
		\rT^{\underline i}_{kl;m}
		=(\rT^{\underline i})_{kl;m}+\rT^{\underline j}_{kl}\,\rGamma^{\underline i}_{\underline{j}m}
	\end{equation}
because
	\[
		\rT^{\underline i}_{;m}s_i=\rnabla_{\partial_m}(\rT^{\underline j}s_j)
		=\rnabla_{\partial_m}\rT^{\underline j}\;s_j+\rT^{\underline j}\,\rnabla_{\partial_m}s_j
		=((\rT^{\underline i})_{;m}+\rT^{\underline j}\,\rGamma^{\underline i}_{\underline{j}m})s_i\;.
	\]
The functions $a^i_j$ and $b^i_j$ also relate the tensor coefficients and $\rnabla$-Cristoffel symbols with respect to $\partial_1,\dots,\partial_n$ and $s_1,\dots,s_n$; for instance,
	\begin{gather}
		\rT^i_{kl;m_1\dots m_r}=\rT^{\underline{\alpha}}_{kl;m_1\dots m_r}b^\alpha_i\;,\quad
		\rT^{\underline{i}}_{kl;m_1\dots m_r}
		=\rT^{\underline{i}}_{\underline{\alpha}\underline{\beta};m_1\dots m_r}a_k^\alpha a_l^\beta\;,
		\label{rT^underline i_kl}\\
		\rT^{\underline{i}}_{\underline{k}\underline{l};m_1\dots m_r}
		=\rT^{\underline{i}}_{\underline{k}\underline{l};
		\underline{\gamma_1}\dots\underline{\gamma_r}}
		a^{\gamma_1}_{m_1}\cdots a^{\gamma_r}_{m_r}\;,
		\label{rT^underline i_kl;m_1 dots m_r}\\	
		\rR^i_{jkl;m_1\dots m_r}
		=\rR^{\underline{\alpha}}_{\underline{\beta}kl;m_1\dots m_r}
		b_i^\alpha a_j^\beta\;,\quad			
		\rR^{\underline{i}}_{\underline{j}kl;m_1\dots m_r}
		=\rR^{\underline{i}}_{\underline{j}\underline{\alpha}\underline{\beta};m_1\dots m_r}
		a_k^\alpha a_l^\beta\;,
		\label{rR^i_jkl}\\
		\rR^{\underline{i}}_{\underline{j}\underline{k}\underline{l};m_1\dots m_r}
		=\rR^{\underline{i}}_{\underline{j}\underline{k}\underline{l};
		\underline{\gamma_1}\dots\underline{\gamma_r}}
		a^{\gamma_1}_{m_1}\cdots a^{\gamma_r}_{m_r}\;,
		\label{rR^i_jkl;m_1 dots m_r}\\
		\rGamma^i_{jk}=\rGamma^{\underline{\alpha}}_{jk}b^i_\alpha\;,\quad
		\rGamma^{\underline{i}}_{jk}=\partial_ka^i_j
		+\rGamma^{\underline{i}}_{\underline{\alpha}k}a^\alpha_j\;.
		\label{rGamma^i_jk}
	\end{gather}

The connection $1$-forms of $\rnabla$, $\theta^{\underline{i}}_{\underline{j}}$ with respect to $s_1,\dots,s_n$ and $\theta^i_j$ with respect to $\partial_1,\dots,\partial_n$, are given by\footnote{The index notation of the connection forms is different from \cite{AtiyahBottPatodi1973}.} $\rnabla s_j=\theta^{\underline{i}}_{\underline{j}}s_i$ and $\rnabla\partial_j=\theta^i_j\partial_i$, using again the non-underlined/underlined index notation; thus $\theta^{\underline{i}}_{\underline{j}}=\rGamma^{\underline i}_{\underline{j}k}\,dx^k$ and $\theta^i_j=\rGamma^i_{jk}\,dx^k$. Observe that $\theta^{\underline{i}}_{\underline{j}}=-\theta^{\underline{j}}_{\underline{i}}$ and $\theta^i_j=-\theta^j_i$ because $\rnabla$ is a metric connection, and 
	\begin{equation}\label{theta^i_j=0 if i le n' and j>n'}
		\theta^{\underline{i}}_{\underline{j}}=\theta^{\underline{j}}_{\underline{i}}=0\quad\text{if}\quad i\le n'\quad\text{and}\quad j>n'
	\end{equation}
because $\rnabla$ preserves the spaces of horizontal and vertical vector fields. 
The following well known expressions can be easily checked, either directly, or by using the structure equations on the principal frame bundle \cite[Chap.~III, Theorem~2.4 and p.~145]{KobayashiNomizu1963-I}:
	\begin{alignat}{2}
		d\theta^i&=-\theta^{\underline{i}}_{\underline{j}}\wedge\theta^j+\rT^{\underline i}\;,&\quad
		0&=-\theta^i_j\wedge dx^j+\rT^i\;,\label{d theta^i}\\
		d\theta^{\underline{i}}_{\underline{j}}&=-\theta^{\underline{i}}_{\underline{k}}\wedge\theta^{\underline{k}}_{\underline{j}}+\rR^{\underline i}_{\underline j}\;,&\quad
		d\theta^i_j&=-\theta^i_k\wedge\theta^k_j+\rR^k_j\;.
		\label{d theta^i_j - theta^i_k wedge theta^k_j}
	\end{alignat}

Consider the normal and leafwise radial vector fields, $\RR'=x^{\prime i}\,\partial'_i$ and $\RR''=x^{\prime\prime i}\,\partial''_i$. Also, let $r'$ and $r''$ be the normal and leafwise radius functions on $U$, determined by ${r'}^2=x^{\prime i}\,x^{\prime i}$ and ${r''}^2=x^{\prime\prime i}\,x^{\prime\prime i}$. It is elementary that
	\begin{alignat*}{3}
		\RR'x^{\prime i}&=x^{\prime i}\;,&\quad\RR'x^{\prime\prime i}&=0\;,&\quad
		\RR'r'&=r'\;,\\
		\RR''x^{\prime\prime i}&=x^{\prime\prime i}\;,&\quad\RR''x^{\prime i}&=0\;,&\quad
		\RR''r''&=r''\;,
	\end{alignat*}
obtaining
	\begin{gather}
		\LL_{\RR'}dx^{\prime i}=dx^{\prime i}\;,\quad
		\LL_{\RR''}dx^{\prime\prime i}=dx^{\prime\prime i}\;,\quad
		\LL_{\RR'}dx^{\prime\prime i}=\LL_{\RR''}dx^{\prime i}=0\;,\label{LL_RR' dx^prime i = dx^prime i}\\
		\RR'(x^{\prime i}/r')=\LL_{\RR'}(dx^{\prime i}/r')
		=\RR''(x^{\prime\prime i}/r'')=\LL_{\RR''}(dx^{\prime\prime i}/r'')=0\;,\label{RR' x^prime i /r' =0}\\
		\LL_{\RR'}(dx^{\prime\prime i}/r')=-dx^{\prime\prime i}/r'\;,\quad
		\LL_{\RR''}(dx^{\prime i}/r'')=-dx^{\prime i}/r''\;.\label{RR'(x^prime i /r')=0}
	\end{gather}
On $B'\times\{0\}$, we have
	\begin{equation}\label{theta^i_j(RR') = 0}
		\theta^{\prime i}(\RR')=x^{\prime i}\;,\quad
		\theta^{\prime\prime i}(\RR')=\theta^{\underline{i}}_{\underline{j}}(\RR')=\theta^i_j(\RR')=0\;,
	\end{equation}
and, on the whole of $U$,
	\begin{equation}\label{theta^i_j(RR'') = 0}
		\theta^{\prime\prime i}(\RR'')=x^{\prime\prime i}\;,\quad
		\theta^{\prime i}(\RR'')=\theta^{\underline{i}}_{\underline{j}}(\RR'')=\theta^i_j(\RR'')=0\;.
	\end{equation}
The first equality of~\eqref{theta^i_j(RR') = 0} holds because, for $x\in B'\times\{0\}$ and $0\le t\le1$, since $r'(tx)=t\,r'(x)$ and $\RR'/r'$ is $\rnabla$-parallel along the geodesics orthogonal to the leaves emanating from $p$, we get
	\[
		(\theta^{\prime i}(\RR'/r'))(x)=(\theta^{\prime i}(\RR'/r'))(tx)
		=\frac{x^{\prime j}}{r'(x)}(\theta^{\prime i}(\partial'_j))(tx)\;,
	\]
which converges to $\frac{x^{\prime i}}{r'(x)}$ as $t\to0$. The same kind of argument proves the first equality of~\eqref{theta^i_j(RR'') = 0}. The other equalities of~\eqref{theta^i_j(RR') = 0} and~\eqref{theta^i_j(RR'') = 0} follow directly from the given definitions.

\begin{prop}\label{p:equations relating Cristoffel symbols and curvature, underlined}
	On $B'\times\{0\}$, we have
		\begin{alignat}{2}
			\RR'\rGamma^{\underline i}_{\underline{j}k}+\rGamma^{\underline i}_{\underline{j}k}
			&=\rR^{\underline i}_{\underline{j}lk}\,x^{\prime l}\quad
			&\text{if}\quad k\le n'\;,\label{RR' rGamma, k le n', underlined}\\
			\RR'\rGamma^{\underline i}_{\underline{j}k}
			&=\rR^{\underline i}_{\underline{j}lk}\,x^{\prime l}\quad
			&\text{if}\quad k>n'\;,\label{RR' rGamma, k>n', underlined}
		\end{alignat}	
	and, on $B'\times B''$,
		\begin{alignat}{2}
			\RR''\rGamma^{\underline i}_{\underline{j}k}
			&=\rR^{\underline i}_{\underline{j}lk}\,x^{\prime\prime l}\quad
			&\text{if}\quad k\le n'\;,\label{RR'' rGamma, k le n', underlined}\\
			\RR''\rGamma^{\underline i}_{\underline{j}k}+\rGamma^{\underline i}_{\underline{j}k}
			&=\rR^{\underline i}_{\underline{j}lk}\,x^{\prime\prime l}\quad
			&\text{if}\quad k>n'\;.\label{RR'' rGamma, k>n', underlined}
		\end{alignat}
\end{prop}

\begin{proof}
	On $B'\times\{0\}$, by~\eqref{theta^i_j(RR') = 0} and~\eqref{d theta^i_j - theta^i_k wedge theta^k_j},
		\begin{equation}\label{LL_RR' theta^i_j}
			\LL_{\RR'}\theta^{\underline{i}}_{\underline{j}}=\iota_{\RR'}d\theta^{\underline{i}}_{\underline{j}}
			=\iota_{\RR'}(d\theta^{\underline{i}}_{\underline{j}}+\theta^{\underline{i}}_{\underline{k}}\wedge\theta^{\underline{k}}_{\underline{j}})
			=\iota_{\RR'}\rR^{\underline i}_{\underline j}
			=\rR^{\underline i}_{\underline{j}kl}x^{\prime k}\,dx^l\;,
		\end{equation}
	and, by~\eqref{LL_RR' dx^prime i = dx^prime i},
		\[
			\LL_{\RR'}\theta^{\underline{i}}_{\underline{j}}=\LL_{\RR'}(\rGamma^{\underline i}_{\underline{j}k}\,dx^k)
			=\RR'\rGamma^{\underline i}_{\underline{j}k}\,dx^k
			+\rGamma^{\underline i}_{\underline{j}k}\,dx^{\prime k}\;,
		\]
	giving~\eqref{RR' rGamma, k le n', underlined} and~\eqref{RR' rGamma, k>n', underlined}. The equalities~\eqref{RR'' rGamma, k le n', underlined} and~\eqref{RR'' rGamma, k>n', underlined} can be proved with similar arguments, using~\eqref{theta^i_j(RR'') = 0},~\eqref{d theta^i_j - theta^i_k wedge theta^k_j} and~\eqref{LL_RR' dx^prime i = dx^prime i}.
\end{proof}

For any $h\in C^\infty(B'\times\{0\})$ and $x\in B'\times\{0\}$, the function $f(t)=h(tx)$ satisfies
	\begin{equation}\label{tf'(t)=RR'h(tx)}
		tf'(t)=tx^{\prime l}\partial'_lh(tx)=\RR'h(tx)\;.
	\end{equation}
Similarly, for any $h\in C^\infty(B'\times B'')$ and $x=(x',x'')\in B'\times\{0\}$, the function $f(t)=h(x',tx'')$ satisfies
	\begin{equation}\label{tf'(t)=RR''h(x',tx'')}
		tf'(t)=tx^{\prime\prime l}\partial''_lh(x',tx'')=\RR''h(x',tx'')\;.
	\end{equation}

\begin{cor}\label{c:equations relating Cristoffel symbols and curvature, underlined}
	For $x\in B'\times\{0\}$, we have 
		\begin{alignat}{2}
			\rGamma^{\underline{i}}_{\underline{j}k}(x)
			&=x^{\prime l}\int_0^1\tau\rR^{\underline{i}}_{\underline{j}lk}(\tau x)\,d\tau\quad
			&\text{if}\quad k\le n'\;,\label{rGamma, B' times 0, k le n'}\\
			\rGamma^{\underline{i}}_{\underline{j}k}(x)
			&=x^{\prime l}\int_0^1\rR^{\underline{i}}_{\underline{j}lk}(\tau x)\,d\tau\quad
			&\text{if}\quad k>n'\;.\label{rGamma, B' times 0, k > n'}\\
		\intertext{For $x=(x',x'')\in B'\times B''$, we have}
			\rGamma^{\underline{i}}_{\underline{j}k}(x)
			&=\rGamma^{\underline{i}}_{\underline{j}k}(x',0)
			+x^{\prime\prime l}\int_0^1\rR^{\underline{i}}_{\underline{j}lk}(x',\tau x'')\,d\tau\quad
			&\text{if}\quad k\le n'\;,\label{rGamma, B' times B'', k le n'}\\
			\rGamma^{\underline{i}}_{\underline{j}k}(x)
			&=x^{\prime\prime l}\int_0^1\tau\rR^{\underline{i}}_{\underline{j}lk}(x',\tau x'')\,d\tau\quad
			&\text{if}\quad k>n'\;.\label{rGamma, B' times B'', k > n'}
		\end{alignat}
\end{cor}

\begin{proof}
	Fix $x\in B'\times\{0\}$. If $k\le n'$, the function $f(t)=t\rGamma^{\underline{i}}_{\underline{j}k}(tx)$ satisfies
		\[
			f'(t)=(\rGamma^{\underline{i}}_{\underline{j}k}
			+\RR'\rGamma^{\underline{i}}_{\underline{j}k})(tx)
			=tx^{\prime l}\rR^{\underline{i}}_{\underline{j}lk}(tx)
		\]
	by~\eqref{tf'(t)=RR'h(tx)} and~\eqref{RR' rGamma, k le n', underlined}, obtaining~\eqref{rGamma, B' times 0, k le n'}. If $k>n'$, the function $f(t)=\rGamma^{\underline{i}}_{\underline{j}k}(tx)$ satisfies
		\[
			f'(t)=\frac{1}{t}\RR'\rGamma^{\underline{i}}_{\underline{j}k}(tx)
			=x^{\prime l}\rR^{\underline{i}}_{\underline{j}lk}(tx)
		\]
	by~\eqref{tf'(t)=RR'h(tx)} and~\eqref{RR' rGamma, k>n', underlined}, obtaining~\eqref{rGamma, B' times 0, k > n'} since $f(0)=0$ by~\eqref{rGamma^underline i_underline j k(0)=0}.
	
	The proofs of~\eqref{rGamma, B' times B'', k le n'} and~\eqref{rGamma, B' times B'', k > n'} are similar, using~\eqref{tf'(t)=RR''h(x',tx'')},~\eqref{RR'' rGamma, k le n', underlined} and~\eqref{RR'' rGamma, k>n', underlined}.
\end{proof}

For any multi-index $I=(i_1,\dots,i_n)$, with $i_r\in\N$, let  
	\begin{alignat*}{2}
		|I|&=i_1+\dots+i_n\;,&\quad
		\partial_I&=\partial_1^{i_1}\cdots\partial_n^{i_n}\;,\\
		|I|'&=i_1+\dots+i_{n'}\;,&\quad
		\partial'_I&=\partial_1^{i_1}\cdots\partial_{n'}^{i_{n'}}\;,\\
		|I|''&=i_{n'+1}+\dots+i_n\;,&\quad
		\partial''_I&=\partial_{n'+1}^{i_{n'+1}}\cdots\partial_n^{i_n}\;.
	\end{alignat*}
By taking derivatives of~\eqref{rGamma, B' times 0, k le n'}--\eqref{rGamma, B' times B'', k > n'}, we get the following.

\begin{cor}\label{c:equations relating the derivatives of Cristoffel symbols and curvature, underlined}
	For $x_0\in B'\times\{0\}$, we have 
		\begin{align*}
			\partial'_I\rGamma^{\underline{i}}_{\underline{j}k}(x_0)
			&=\int_0^1\tau^{|I|'}
			(\partial'_I(x\mapsto x^{\prime l}\rR^{\underline{i}}_{\underline{j}lk}(x)))(\tau x_0)\,d\tau\\
		\intertext{if $k\le n'$, and}
			\partial'_I\rGamma^{\underline{i}}_{\underline{j}k}(x_0)
			&=\int_0^1\tau^{|I|'-1}
			(\partial'_I(x\mapsto x^{\prime l}\rR^{\underline{i}}_{\underline{j}lk}(x)))(\tau x_0)\,d\tau\\
		\intertext{if $k>n'$ and $|I|'>0$. For $x_0=(x_0',x_0'')\in B'\times B''$, we have}
			\partial_I\rGamma^{\underline{i}}_{\underline{j}k}(x_0)
			&=\partial'_I\rGamma^{\underline{i}}_{\underline{j}k}(x_0',0)
			+x_0^{\prime\prime l}\int_0^1
			(\partial_I(x\mapsto\rR^{\underline{i}}_{\underline{j}lk}(x)))(x_0',\tau x_0'')\,d\tau\\
		\intertext{if $k\le n'$ and $|I|''=0$,}
			\partial_I\rGamma^{\underline{i}}_{\underline{j}k}(x_0)
			&=\partial'_I\rGamma^{\underline{i}}_{\underline{j}k}(x_0',0)\notag\\
			&\phantom{=\text{}}\text{}+\int_0^1\tau^{|I|''-1}
			(\partial_I(x\mapsto x^{\prime\prime l}\rR^{\underline{i}}_{\underline{j}lk}(x)))
			(x_0',\tau x_0'')\,d\tau\\
		\intertext{if $k\le n'$ and $|I|''>0$, and}
			\partial_I\rGamma^{\underline{i}}_{\underline{j}k}(x_0)
			&=\int_0^1\tau^{|I|''}
			(\partial_I(x\mapsto x^{\prime\prime l}\rR^{\underline{i}}_{\underline{j}lk}(x)))(x_0',\tau x_0'')\,d\tau
		\end{align*}
	if $k>n'$.
\end{cor}

The following result is a version with non-underlined indices of Proposition~\ref{p:equations relating Cristoffel symbols and curvature, underlined}, and follows with the same arguments.

\begin{prop}\label{p:equations relating Cristoffel symbols and curvature, non-underlined}
	On $B'\times\{0\}$, we have
		\begin{alignat*}{2}
			\RR'\rGamma^i_{jk}+\rGamma^i_{jk}
			&=\rR^i_{jlk}\,x^{\prime l}\quad
			&\text{if}\quad k\le n'\;,\\
			\RR'\rGamma^i_{jk}
			&=\rR^i_{jlk}\,x^{\prime l}\quad
			&\text{if}\quad k>n'\;,
		\end{alignat*}	
	and, on $B'\times B''$,
		\begin{alignat*}{2}
			\RR''\rGamma^i_{jk}
			&=\rR^i_{jlk}\,x^{\prime\prime l}\quad
			&\text{if}\quad k\le n'\;,\\
			\RR''\rGamma^i_{jk}+\rGamma^i_{jk}
			&=\rR^i_{jlk}\,x^{\prime\prime l}\quad
			&\text{if}\quad k>n'\;.
		\end{alignat*}
\end{prop}

A direct consequence of Proposition~\ref{p:equations relating Cristoffel symbols and curvature, non-underlined} is
	\begin{gather*}
		\rGamma^i_{jk}(0)=0\;,\\
		\rGamma^i_{jk}(x',0)=0\quad\forall x'\in B'\quad\text{if}\quad k>n'\;.
	\end{gather*}
Then the following corollaries of Proposition~\ref{p:equations relating Cristoffel symbols and curvature, non-underlined} can be proved like Corollaries~\ref{c:equations relating Cristoffel symbols and curvature, underlined} and~\ref{c:equations relating the derivatives of Cristoffel symbols and curvature, underlined}.

\begin{cor}\label{c:equations relating Cristoffel symbols and curvature, non-underlined}
	For $x\in B'\times\{0\}$, we have 
		\begin{alignat*}{2}
			\rGamma^i_{jk}(x)
			&=x^{\prime l}\int_0^1\tau\rR^i_{jlk}(\tau x)\,d\tau\quad
			&\text{if}\quad k\le n'\;,\\
			\rGamma^i_{jk}(x)
			&=x^{\prime l}\int_0^1\rR^i_{jlk}(\tau x)\,d\tau\quad
			&\text{if}\quad k>n'\;.
		\intertext{For $x=(x',x'')\in B'\times B''$, we have}
			\rGamma^i_{jk}(x)
			&=\rGamma^i_{jk}(x',0)
			+x^{\prime\prime l}\int_0^1\rR^i_{jlk}(x',\tau x'')\,d\tau\quad
			&\text{if}\quad k\le n'\;,\\
			\rGamma^i_{jk}(x)
			&=x^{\prime\prime l}\int_0^1\tau\rR^i_{jlk}(x',\tau x'')\,d\tau\quad
			&\text{if}\quad k>n'\;.
		\end{alignat*}
\end{cor}

\begin{cor}\label{c:equations relating the derivatives of Cristoffel symbols and curvature, non-underlined}
	For $x_0\in B'\times\{0\}$, we have 
		\begin{align*}
			\partial'_I\rGamma^i_{jk}(x_0)
			&=\int_0^1\tau^{|I|'}
			(\partial'_I(x\mapsto x^{\prime l}\rR^i_{jlk}(x)))(\tau x_0)\,d\tau\\
		\intertext{if $k\le n'$ and $|I|'>0$, and}
			\partial'_I\rGamma^i_{jk}(x_0)
			&=\int_0^1\tau^{|I|'-1}
			(\partial'_I(x\mapsto x^{\prime l}\rR^i_{jlk}(x)))(\tau x_0)\,d\tau\\
		\intertext{if $k>n'$. For $x_0=(x_0',x_0'')\in B'\times B''$, we have}
			\partial_I\rGamma^i_{jk}(x_0)
			&=\partial'_I\rGamma^i_{jk}(x_0',0)
			+x_0^{\prime\prime l}\int_0^1
			(\partial_I(x\mapsto\rR^i_{jlk}(x)))(x_0',\tau x_0'')\,d\tau\\
		\intertext{if $k\le n'$ and $|I|''=0$,}
			\partial_I\rGamma^i_{jk}(x_0)
			&=\partial'_I\rGamma^i_{jk}(x_0',0)\notag\\
			&\phantom{=\text{}}\text{}+\int_0^1\tau^{|I|''-1}
			(\partial_I(x\mapsto x^{\prime\prime l}\rR^i_{jlk}(x)))(x_0',\tau x_0'')\,d\tau\\
		\intertext{if $k\le n'$ and $|I|''>0$, and}
			\partial_I\rGamma^i_{jk}(x_0)
			&=\int_0^1\tau^{|I|''}
			(\partial_I(x\mapsto x^{\prime\prime l}\rR^i_{jlk}(x)))(x_0',\tau x_0'')\,d\tau
		\end{align*}
	if $k>n'$.
\end{cor}

Let
	\[
		F^i_{jkl}
		=\rGamma^m_{kl}\rT^{\underline i}_{mj}
		-\rGamma^m_{kj}\rT^{\underline i}_{ml}+(\rT^{\underline i})_{kj;l}
		-(\rT^{\underline i})_{kl;j}+
		\rT^m_{lj}\,\rT^{\underline i}_{km}\;.
	\]

\begin{prop}\label{p:equations relating a^i_j, curvature and torsion}
	On $B'\times\{0\}$, we have
		\begin{alignat}{2}
			({\RR'}^2+\RR')a^i_j
			&=(\rR^{\underline i}_{\underline{k}lj}+F^i_{jkl})x^{\prime l}x^{\prime k}
			+\rT^{\underline i}_{kj}x^{\prime k}&\quad\text{if}\quad &i,j\le n'\;,
			\label{RR' ^2 a^i_j+RR' a^i_j, i,j le n'}\\
			({\RR'}^2+\RR')a^i_j
			&=F^i_{jkl}x^{\prime l}x^{\prime k}
			+\rT^{\underline i}_{kj}x^{\prime k}&\quad\text{if}\quad &j\le n'<i\;,
			\label{RR' ^2 a^i_j+RR' a^i_j, j le n' < i}\\
			({\RR'}^2-\RR')a^i_j
			&=(\rR^{\underline i}_{\underline{k}lj}
			+F^i_{jkl})x^{\prime l}x^{\prime k}&\quad\text{if}\quad &i\le n'<j\;,
			\label{RR' ^2 a^i_j-RR' a^i_j, i le n' < j}\\
			({\RR'}^2-\RR')a^i_j
			&=F^i_{jkl}x^{\prime l}x^{\prime k}&\quad\text{if}\quad &n'<i,j\;. 
			\label{RR' ^2 a^i_j-RR' a^i_j, n' < i,j}\\
		\intertext{On $B'\times B''$, we have}
			({\RR''}^2-\RR'')a^i_j
			&=F^i_{jkl}x^{\prime\prime l}x^{\prime\prime k}&\quad\text{if}\quad&i,j\le n'\;,
			\label{RR'' ^2 a^i_j-RR'' a^i_j, i,j le n'}\\
			({\RR''}^2-\RR'')a^i_j
			&=(\rR^{\underline i}_{\underline{k}lj}
			+F^i_{jkl})x^{\prime\prime l}x^{\prime\prime k}&\quad\text{if}\quad&j\le n'<i\;,
			\label{RR'' ^2 a^i_j-RR'' a^i_j, j le n' < i}\\
			({\RR''}^2+\RR'')a^i_j
			&=F^i_{jkl}x^{\prime\prime l}x^{\prime\prime k}&\quad\text{if}\quad&i\le n'<j\;,
			\label{RR'' ^2 a^i_j+RR'' a^i_j, i le n' < j}\\
			({\RR''}^2+\RR'')a^i_j
			&=(\rR^{\underline i}_{\underline{k}lj}
			+F^i_{jkl})x^{\prime\prime l}x^{\prime\prime k}&\quad\text{if}\quad&n'<i,j\;.
			\label{RR'' ^2 a^i_j+RR'' a^i_j, n' < i,j}
		\end{alignat}
\end{prop}

\begin{proof}
	We only prove~\eqref{RR' ^2 a^i_j+RR' a^i_j, i,j le n'}--\eqref{RR' ^2 a^i_j-RR' a^i_j, n' < i,j} because the proof of~\eqref{RR'' ^2 a^i_j-RR'' a^i_j, i,j le n'}--\eqref{RR'' ^2 a^i_j+RR'' a^i_j, n' < i,j} is analogous, observing that $\rT^{\underline i}_{kj}=0$ for $k,j>n'$ by~\eqref{rT(V,W)}. 
	
	On $B'\times\{0\}$,
		\begin{multline*}
			\LL_{\RR'}\theta^{\prime i}=\iota_{\RR'}d\theta^{\prime i}+d\iota_{\RR'}\theta^{\prime i}\\
			=\iota_{\RR'}(-\theta^{\underline{i}}_{\underline{k}}\wedge\theta^k+\rT^{\underline i})+dx^{\prime i}
			=(\theta^{\underline{i}}_{\underline{k}}+\iota_{\partial'_k}\rT^{\underline i})x^{\prime k}+dx^{\prime i}
		\end{multline*}
	 by~\eqref{d theta^i} and~\eqref{theta^i_j(RR') = 0}, and
	 	\[
			\LL_{\RR'}\theta^{\prime\prime i}
			=\iota_{\RR'}d\theta^{\prime\prime i}+d\iota_{\RR'}\theta^{\prime\prime i}
			=\iota_{\RR'}(-\theta^{\underline{i}}_{\underline{k}}\wedge\theta^k+\rT^{\underline i})
			=\iota_{\partial'_k}\rT^{\underline i}x^{\prime k}
		\]
	by~\eqref{d theta^i},~\eqref{theta^i_j(RR') = 0} and~\eqref{theta^i_j=0 if i le n' and j>n'}. Hence, by~\eqref{RR'(x^prime i /r')=0} and~\eqref{LL_RR' theta^i_j},
		\begin{align}
			r'\,\LL_{\RR'}\,\frac{1}{r'}\,\LL_{\RR'}\theta^{\prime i}
			&=(\LL_{\RR'}\theta^{\underline{i}}_{\underline{k}}
			+\LL_{\RR'}\iota_{\partial'_k}\rT^{\underline i})x^{\prime k}\notag\\
			&=(\rR^{\underline i}_{\underline{k}lj}x^{\prime l}\,dx^j
			+\iota_{\RR'}d\iota_{\partial'_k}\rT^{\underline i}
			+d\iota_{\RR'}\iota_{\partial'_k}\rT^{\underline i})x^{\prime k}\;,
			\label{1st expression of r' LL_RR' frac 1 r' LL_RR' theta^prime i}\\
			r'\,\LL_{\RR'}\,\frac{1}{r'}\,\LL_{\RR'}\theta^{\prime\prime i}
			&=\LL_{\RR'}\iota_{\partial'_k}\rT^{\underline i}\;x^{\prime k}
			=(\iota_{\RR'}d\iota_{\partial'_k}\rT^{\underline i}
			+d\iota_{\RR'}\iota_{\partial'_k}\rT^{\underline i})x^{\prime k}\;.
			\label{1st expression of r' LL_RR' frac 1 r' LL_RR' theta^prime prime i}
		\end{align}
	Since $\breve\nabla$ is torsion free, we have \cite[Proposition~4.1]{Poor1981}
		\begin{equation}\label{1st expression of d iota_partial'_k rT^underline i}
			d\iota_{\partial'_k}\rT^{\underline i}
			=dx^j\wedge\breve\nabla_{\partial_j}(\iota_{\partial'_k}\rT^{\underline i})
			=dx^j\wedge(\iota_{\breve\nabla_{\partial_j}\partial'_k}\rT^{\underline i}
			+\iota_{\partial'_k}\breve\nabla_{\partial_j}\rT^{\underline i})\;.
		\end{equation}
	But
		\[
			\breve\nabla_{\partial_j}\partial'_k=\rnabla_{\partial_j}\partial'_k-\frac{1}{2}\rT(\partial_j,\partial'_k)
			=(\rGamma^m_{kj}-\rT^m_{jk}/2)\partial_m\;,
		\]
	and
		\begin{align*}
			(\breve\nabla_{\partial_j}\rT^{\underline i})(\partial_\alpha,\partial_\beta)
			&=\partial_j(\rT^{\underline i}(\partial_\alpha,\partial_\beta))
			-\rT^{\underline i}(\breve\nabla_{\partial_j}\partial_\alpha,\partial_\beta)
			-\rT^{\underline i}(\partial_\alpha,\breve\nabla_{\partial_j}\partial_\beta)\\
			&=\partial_j(\rT^{\underline i}(\partial_\alpha,\partial_\beta))
			-\rT^{\underline i}(\rnabla_{\partial_j}\partial_\alpha,\partial_\beta)
			-\rT^{\underline i}(\partial_\alpha,\rnabla_{\partial_j}\partial_\beta)\\
			&\phantom{=\text{}}\text{}
			+\frac{1}{2}\left(\rT^{\underline i}(\rT(\partial_j,\partial_\alpha),\partial_\beta)
			+\rT^{\underline i}(\partial_\alpha,\rT(\partial_j,\partial_\beta))\right)\\
			&=(\rnabla_{\partial_j}\rT^{\underline i})(\partial_\alpha,\partial_\beta)
			+(\rT^m_{j\alpha}\,\rT^{\underline i}_{m\beta}
			+\rT^m_{j\beta}\,\rT^{\underline i}_{\alpha m})/2\\
			&=(\rT^{\underline i})_{\alpha\beta;j}
			+(\rT^m_{j\alpha}\,\rT^{\underline i}_{m\beta}
			+\rT^m_{j\beta}\,\rT^{\underline i}_{\alpha m})/2\;,
		\end{align*}
	yielding
		\[
			\breve\nabla_{\partial_j}\rT^{\underline i}
			=\frac{1}{2}\left((\rT^{\underline i})_{\alpha\beta;j}
			+(\rT^m_{j\alpha}\,\rT^{\underline i}_{m\beta}
			+\rT^m_{j\beta}\,\rT^{\underline i}_{\alpha m})/2\right)dx^\alpha\wedge dx^\beta\;.
		\]
	So
		\begin{align*}
			\iota_{\breve\nabla_{\partial_j}\partial'_k}\rT^{\underline i}
			&=(\rGamma^m_{kj}-\rT^m_{jk}/2)\rT^{\underline i}_{m\beta}\,dx^\beta\;,\\
			\iota_{\partial'_k}\breve\nabla_{\partial_j}\rT^{\underline i}
			&=\left((\rT^{\underline i})_{k\beta;j}
			+(\rT^m_{jk}\,\rT^{\underline i}_{m\beta}
			+\rT^m_{j\beta}\,\rT^{\underline i}_{km})/2\right)dx^\beta\;,
		\end{align*}
	obtaining
		\begin{equation}\label{d iota_partial'_k rT^underline i}
			d\iota_{\partial'_k}\rT^{\underline i}
			=\left(\rGamma^m_{kj}\rT^{\underline i}_{m\beta}+(\rT^{\underline i})_{k\beta;j}
			+\rT^m_{j\beta}\,\rT^{\underline i}_{km}/2\right)dx^j\wedge dx^\beta
		\end{equation}
	by~\eqref{1st expression of d iota_partial'_k rT^underline i}, and therefore
		\begin{equation}\label{iota_RR' d iota_partial'_k rT^underline i}
			\iota_{\RR'}d\iota_{\partial'_k}\rT^{\underline i}=F^i_{jkl}x^{\prime l}\,dx^j\;.
		\end{equation}
	Moreover
		\begin{equation}\label{d iota_RR' iota_partial'_k rT^underline i x^prime k}
			d\iota_{\RR'}\iota_{\partial'_k}\rT^{\underline i}\;x^{\prime k}
			=-\iota_{\RR'}\iota_{\partial'_k}\rT^{\underline i}\,dx^{\prime k}
			=-\rT^{\underline i}_{kl}x^{\prime l}\,dx^{\prime k}
			=\rT^{\underline i}_{lk}x^{\prime l}\,dx^{\prime k}
		\end{equation}
	because $\iota_{\RR'}\iota_{\partial'_k}\rT^{\underline i}\;x^{\prime k}=\iota_{\RR'}\iota_{\RR'}\rT^{\underline i}=0$.
	Then, by~\eqref{1st expression of r' LL_RR' frac 1 r' LL_RR' theta^prime i},~\eqref{1st expression of r' LL_RR' frac 1 r' LL_RR' theta^prime prime i},~\eqref{iota_RR' d iota_partial'_k rT^underline i} and~\eqref{d iota_RR' iota_partial'_k rT^underline i x^prime k},
		\begin{align}
			r'\,\LL_{\RR'}\,\frac{1}{r'}\,\LL_{\RR'}\theta^{\prime i}
			&=((\rR^{\underline i}_{\underline{k}lj}+F^i_{jkl})x^{\prime l}
			+\rT^{\underline i}_{kj})x^{\prime k}\,dx^{\prime j}\notag\\
			&\phantom{=\text{}}\text{}+(\rR^{\underline i}_{\underline{k}lj}
			+F^i_{jkl})x^{\prime l}x^{\prime k}\,dx^{\prime\prime j}\;,
			\label{r' LL_RR' frac 1 r' LL_RR' theta^prime i}\\
			r'\,\LL_{\RR'}\,\frac{1}{r'}\,\LL_{\RR'}\theta^{\prime\prime i}
			&=(F^i_{jkl}x^{\prime l}
			+\rT^{\underline i}_{kj})x^{\prime k}\,dx^{\prime j}\notag\\
			&\phantom{=\text{}}\text{}+F^i_{jkl}x^{\prime l}x^{\prime k}\,dx^{\prime\prime j}\;.
			\label{r' LL_RR' frac 1 r' LL_RR' theta^prime prime i}
		\end{align}
	On the other hand,
		\begin{multline*}
			\LL_{\RR'}\theta^i=\LL_{\RR'}(a^i_j\,dx^j)=\RR'a^i_j\;dx^j+a^i_j\LL_{\RR'}dx^j\\
			=\RR'a^i_j\;dx^j+a^i_j\,dx^{\prime j}=(\RR'a^i_j+a^i_j)\,dx^{\prime j}+\RR'a^i_j\,dx^{\prime\prime j}
		\end{multline*}
	by~\eqref{LL_RR' dx^prime i = dx^prime i}, yielding
		\begin{equation}\label{r' LL_RR' frac 1 r' LL_RR' theta^i}
			r'\,\LL_{\RR'}\,\frac{1}{r'}\,\LL_{\RR'}\theta^i=({\RR'}^2a^i_j+\RR'a^i_j)\,dx^{\prime j}
			+({\RR'}^2a^i_j-\RR'a^i_j)\,dx^{\prime\prime j}
		\end{equation}
	by~\eqref{RR' x^prime i /r' =0} and~\eqref{RR'(x^prime i /r')=0}. Now~\eqref{RR' ^2 a^i_j+RR' a^i_j, i,j le n'}--\eqref{RR' ^2 a^i_j-RR' a^i_j, n' < i,j} follow by equating the corresponding coefficients in~\eqref{r' LL_RR' frac 1 r' LL_RR' theta^prime i}--\eqref{r' LL_RR' frac 1 r' LL_RR' theta^i}.
\end{proof}

Given $|t|\le1$, let $\xi_{ij}(t,x)=a^i_j(tx)$ for $x\in B'\times\{0\}$, and $\hat\xi_{ij}(t,x)=a^i_j(x',tx'')$ for $x=(x',x'')\in B'\times B''$. For functions $f=f(t,x)$, the derivatives with respect to $t$ are denoted by $f'$.

\begin{cor}\label{c:xi_ij'(t,x)}
	For $x\in B'\times\{0\}$, we have
		\begin{align}
			\xi_{ij}'(t,x)&=tx^{\prime l}x^{\prime k}\int_0^1u^2
			(\rR^{\underline i}_{\underline{k}lj}+F^i_{jkl})(tux)\,du
			+x^{\prime k}\int_0^1u\rT^{\underline{i}}_{kj}(tux)\,du
			\label{xi_ij'(t,x), B' times 0, i,j le n'}\\
		\intertext{if $i,j\le n'$,}
			\xi_{ij}'(t,x)&=tx^{\prime l}x^{\prime k}\int_0^1u^2
			F^i_{jkl}(tux)\,du
			+x^{\prime k}\int_0^1u\rT^{\underline{i}}_{kj}(tux)\,du
			\label{xi_ij'(t,x), B' times 0, j le n' < i}\\
		\intertext{if $j\le n'<i$,}
			\xi_{ij}'(t,x)&=\xi_{ij}'(0,x)+tx^{\prime l}x^{\prime k}\int_0^1(\rR^{\underline i}_{\underline{k}lj}
			+F^i_{jkl})(tux)\,du
			\label{xi_ij'(t,x), B' times 0, i le n' < j}\\
		\intertext{if $i\le n'<j$, and}
			\xi_{ij}'(t,x)&=\xi_{ij}'(0,x)+tx^{\prime l}x^{\prime k}\int_0^1F^i_{jkl}(tux)\,du
			\label{xi_ij'(t,x), B' times 0, n' < i,j}\\
		\intertext{if $n'<i,j$. For $x=(x',x'')\in B'\times B''$, we have}
			\hat\xi_{ij}'(t,x)&=\hat\xi_{ij}'(0,x)+tx^{\prime\prime l}x^{\prime\prime k}\int_0^1
			F^i_{jkl}(x',tux'')\,du
			\label{zeta_ij'(t,x), B' times B'', i,j le n'}\\
		\intertext{if $i,j\le n'$,}
			\hat\xi_{ij}'(t,x)&=\hat\xi_{ij}'(0,x)+tx^{\prime\prime l}x^{\prime\prime k}\int_0^1
			(\rR^{\underline i}_{\underline{k}lj}+F^i_{jkl})(x',tux'')\,du
			\label{zeta_ij'(t,x), B' times B'', j le n' < i}\\
		\intertext{if $j\le n'<i$,}
			\hat\xi_{ij}'(t,x)&=tx^{\prime\prime l}x^{\prime\prime k}\int_0^1u^2F^i_{jkl}(x',tux'')\,du
			\label{zeta_ij'(t,x), B' times B'', i le n' < j}\\
		\intertext{if $i\le n'<j$, and}
			\hat\xi_{ij}'(t,x)&=tx^{\prime\prime l}x^{\prime\prime k}\int_0^1u^2
			(\rR^{\underline i}_{\underline{k}lj}+F^i_{jkl})(x',tux'')\,du
			\label{zeta_ij'(t,x), B' times B'', n' < i,j}
		\end{align}
	if $n'<i,j$.
\end{cor}

\begin{proof}
	Let us prove~\eqref{xi_ij'(t,x), B' times 0, j le n' < i} and~\eqref{xi_ij'(t,x), B' times 0, n' < i,j}. By using~\eqref{tf'(t)=RR'h(tx)} with $h=a^i_j$ and $h=\RR'a^i_j$, we get 
		\begin{equation}\label{t xi_ij'(t,x)}
			t\xi_{ij}'(t,x)=\RR'a^i_j(tx)\;,\quad t^2\xi_{ij}''(t,x)+t\xi_{ij}'(t,x)={\RR'}^2a^i_j(tx)\;.
		\end{equation}
	In the case $j\le n'<i$, we get
		\begin{multline*}
			(t^2\xi_{ij}')'(t,x)=t^2\xi_{ij}''(t,x)+2t\xi_{ij}'(t,x)=({\RR'}^2+\RR')a^i_j(tx)\\
			=t^2F_{ijkl}(tx)x^{\prime l}x^{\prime k}+t\rT^{\underline{i}}_{kj}(tx)x^{\prime k}
		\end{multline*}
	by~\eqref{t xi_ij'(t,x)} and~\eqref{RR' ^2 a^i_j+RR' a^i_j, j le n' < i}, obtaining~\eqref{xi_ij'(t,x), B' times 0, j le n' < i} by integration:
		\begin{align*}
			t^2\xi_{ij}'(t,x)&=x^{\prime l}x^{\prime k}\int_0^t\tau^2F_{ijkl}(\tau x)\,d\tau
			+x^{\prime k}\int_0^t\tau\rT^{\underline{i}}_{kj}(\tau x)\,d\tau\\
			&=t^3x^{\prime l}x^{\prime k}\int_0^1u^2F_{ijkl}(tux)\,du
			+t^2x^{\prime k}\int_0^1u\rT^{\underline{i}}_{kj}(tux)\,du\;.
		\end{align*}
	In the case $n'<i,j$, we get
		\[
			t^2\xi_{ij}''(t,x)=({\RR'}^2-\RR')a^i_j(tx)=t^2F_{ijkl}(tx)x^{\prime l}x^{\prime k}
		\]
	by~\eqref{t xi_ij'(t,x)} and~\eqref{RR' ^2 a^i_j-RR' a^i_j, n' < i,j}, obtaining~\eqref{xi_ij'(t,x), B' times 0, n' < i,j} by integration:
		\[
			\xi_{ij}'(t,x)-\xi_{ij}'(0,x)=x^{\prime l}x^{\prime k}\int_0^tF_{ijkl}(\tau x)\,d\tau
			=tx^{\prime l}x^{\prime k}\int_0^1F_{ijkl}(tux)\,du\;.
		\] 
	
	 The other equalities of the statement have similar proofs, using the corresponding equalities of Proposition~\ref{p:equations relating a^i_j, curvature and torsion}, and using~\eqref{tf'(t)=RR''h(x',tx'')} instead of~\eqref{tf'(t)=RR'h(tx)} to get~\eqref{zeta_ij'(t,x), B' times B'', i,j le n'}--\eqref{zeta_ij'(t,x), B' times B'', n' < i,j}.
\end{proof}

\begin{cor}\label{c:(partial_alpha xi_ij)'(t,x)}
	There are polynomials
		\begin{itemize}
			
			\item $P_{Iij}^{\alpha\beta}$ {\rm(}respectively, $\widehat{P}_{Iij}^{\alpha\beta}${\rm)} {\rm(}depending only on $I$, $i$, $j$, $\alpha$ and $\beta${\rm)} in $t$, $u$, $x$, and in the value of the functions $\rR^{\underline{*}}_{\underline{*}\underline{*}\underline{*}}$, $\rT^{\underline{*}}_{\underline{*}\underline{*}}$, $\rT^{\underline{*}}_{\underline{*}\underline{*};\underline{*}}$, $\rGamma^{\underline{*}}_{\underline{*}\underline{*}}$, $\rGamma^*_{**}$ and $a^*_*$ at $tux$ {\rm(}respectively, $(x',tux'')${\rm)}; and
			
			\item $Q_{Iij}$ {\rm(}respectively, $\widehat{Q}_{Iij}${\rm)} {\rm(}depending only on $I$, $i$ and $j${\rm)} in $t$, $x$, and in the value of the partial derivatives up to order $|I|-1$ of $\rR^{\underline{*}}_{\underline{*}\underline{*}\underline{*}}$, $\rT^{\underline{*}}_{\underline{*}\underline{*}}$, $\rT^{\underline{*}}_{\underline{*}\underline{*};\underline{*}}$, $\rGamma^{\underline{*}}_{\underline{*}\underline{*}}$, $\rGamma^*_{**}$ and $a^*_*$ at $tux$ {\rm(}respectively, $(x',tux'')${\rm)},
			
		\end{itemize}
	such that the following properties hold. For $x\in B'\times\{0\}$, we have
		\begin{align*}
			(\partial'_I\xi_{ij})'(t,x)&=\int_0^1P_{Iij}^{\alpha\beta}\,\partial'_I\xi_{\alpha\beta}(tu,x)\,du
			+\int_0^1Q_{Iij}\,du\\
		\intertext{if $j\le n'$, and}
			(\partial'_I\xi_{ij})'(t,x)&=(\partial'_I\xi_{ij})'(0,x)
			+\int_0^1P_{Iij}^{\alpha\beta}\,\partial'_I\xi_{\alpha\beta}(tu,x)\,du
			+\int_0^1Q_{Iij}\,du\\
		\intertext{if $j>n'$. For $x\in B'\times B''$, we have}
			(\partial_I\hat\xi_{ij})'(t,x)
			&=\partial'_I\hat\xi_{ij}'(0,x)+
			\int_0^1\widehat{P}_{Iij}^{\alpha\beta}\,\partial_I\hat\xi_{\alpha\beta}(tu,x)\,du
			+\int_0^1\widehat{Q}_{Iij}\,du\\
		\intertext{if $j\le n'$, and}
			(\partial_I\hat\xi_{ij})'(t,x)
			&=\int_0^1\widehat{P}_{Iij}^{\alpha\beta}\,\partial_I\hat\xi_{\alpha\beta}(tu,x)\,du
			+\int_0^1\widehat{Q}_{Iij}\,du
		\end{align*}
	if $j>n'$. 
\end{cor}

\begin{proof}
	This follows by taking partial derivatives of the equalities of Corollary~\ref{c:xi_ij'(t,x)}, using the expressions of $\rR^{\underline i}_{\underline{k}lj}$ and $\rT^{\underline{i}}_{kj}$ given in~\eqref{rR^i_jkl} and~\eqref{rT^underline i_kl}, and since
		\begin{multline*}
			F^i_{jkl}=\rGamma^m_{kl}\rT^{\underline{i}}_{\underline{\alpha}\underline{\delta}}
			a^\alpha_ma^\beta_j
			+\rGamma^m_{kj}\rT^{\underline{i}}_{\underline{\alpha}\underline{\delta}}
			a^\alpha_ma^\beta_l
			+(\rT^{\underline{i}}_{\underline{\alpha}\underline{\beta};\underline{\gamma}}
			a^\gamma_l
			-\rT^{\underline{m}}_{\underline{\alpha}\underline{\beta}}
			\rGamma^{\underline{i}}_{\underline{m}l})
			a^\alpha_ka^\beta_j\\
			\text{}-(\rT^{\underline{i}}_{\underline{\alpha}\underline{\beta};\underline{\gamma}}a^\gamma_j
			-\rT^{\underline{m}}_{\underline{\alpha}\underline{\beta}}
			\rGamma^{\underline{i}}_{\underline{m}j})a^\alpha_ka^\beta_l
			+\rT^{\underline{\gamma}}_{\underline{\alpha}\underline{\beta}}
			\rT^{\underline{i}}_{\underline{\delta}\underline{\gamma}}a^\alpha_la^\beta_ja^\delta_k
		\end{multline*}
	by~\eqref{rT^underline i_;m} and~\eqref{rT^underline i_kl}.
\end{proof}

\begin{rem}
	From~\eqref{rT^underline i_kl} and~\eqref{rGamma^i_jk}, we can get an expression of $F^i_{jkl}$ involving the Christoffel symbols $\rGamma^{\underline{*}}_{\underline{*}*}$ (with two underlined coefficients). But this expression also involves derivatives of $a^*_*$. For this reason, the functions $\rGamma^*_{**}$ are used in Corollary~\ref{c:(partial_alpha xi_ij)'(t,x)}. Thus Corollaries~\ref{c:equations relating Cristoffel symbols and curvature, non-underlined} and~\ref{c:equations relating the derivatives of Cristoffel symbols and curvature, non-underlined} will be needed to find bounds of the partial derivatives of $\rGamma^*_{**}$, so that Corollary~\ref{c:(partial_alpha xi_ij)'(t,x)} can be applied to find bounds of the partial derivatives of $a_*^*$.
\end{rem}

Let us use the multi-index notation also for covariant derivatives; for example, for any multi-index $I=(i_1,\dots,i_n)$, let
	\begin{align*}
		\rnabla_{\partial_1}^{i_1}\cdots\rnabla_{\partial_n}^{i_n}\rR
		&=\frac{1}{2}\rR^i_{jkl;I}\,dx^k\wedge dx^l\otimes\partial_i\otimes dx^j\\
		&=\frac{1}{2}\rR^{\underline{i}}_{\underline{j}kl;I}\,dx^k\wedge dx^l\otimes s_i\otimes\theta^j\;,\\
		\rnabla_{\partial_1}^{i_1}\cdots\rnabla_{\partial_n}^{i_n}\rT
		&=\frac{1}{2}\rT^i_{kl;I}\,dx^k\wedge dx^l\otimes\partial_i\\
		&=\frac{1}{2}\rT^{\underline{i}}_{kl;I}\,dx^k\wedge dx^l\otimes s_i\;,
	\end{align*}
with 
	\begin{alignat*}{2}
		\rR^i_{jkl;I}&=-\rR^i_{jlk;I}\;,&\quad
		\rR^{\underline{i}}_{\underline{j}kl;I}&=-\rR^{\underline{i}}_{\underline{j}lk;I}\;,\\
		\rT^i_{kl;I}&=-\rT^i_{lk;I}\;,&\quad\rT^{\underline{i}}_{kl;I}&=-\rT^{\underline{i}}_{lk;I}\;.
	\end{alignat*}
Obviously,
	\begin{equation}\label{rR^i_jkl;I}
		\rR^i_{jkl;I}
			=\partial_I\rR^i_{jkl}+P_{Iijkl}\;,\quad
			\rT^i_{kl;I}=\partial_I\rT^i_{kl}+Q_{Iikl}\;,
	\end{equation}
where $P_{Ijkl}$ (respectively, $Q_{Ikl}$) is a polynomial in the partial derivatives up to order $|I|-1$ of the functions $\rR^*_{***}$ (respectively, $\rT^*_{**}$) and $\rGamma^*_{**}$, which depends only on $I$, $i$, $j$, $k$ and $l$ {\rm(}respectively, $I$, $i$, $k$ and $l${\rm)}.

\begin{lem}\label{l:rR^underline i_underline j kl;I}
	We have
		\[
			\rR^{\underline{i}}_{\underline{j}kl;I}
			=\partial_I\rR^{\underline{i}}_{\underline{j}\underline{k}\underline{l}}+P_{Iijkl}\;,\quad
			\rT^{\underline{i}}_{kl;I}=\partial_I\rT^{\underline{i}}_{\underline{k}\underline{l}}+Q_{Iikl}\;,
		\]
	where $P_{Ijkl}$ {\rm(}respectively, $Q_{Ikl}${\rm)} is a polynomial in the partial derivatives up to order $|I|-1$ of the functions $\rR^{\underline{*}}_{\underline{*}\underline{*}\underline{*}}$ {\rm(}respectively, $\rT^{\underline{*}}_{\underline{*}\underline{*}}${\rm)}, $\rGamma^{\underline{*}}_{\underline{*}*}$ and $a^*_*$, which depends only on $I$, $i$, $j$, $k$ and $l$ {\rm(}respectively, $I$, $i$, $k$ and $l${\rm)}.
\end{lem}

\begin{proof}
	The first equality is a version of \cite[Lemma~2.18]{Schick2001}, which follows with the same arguments, and the second equality has a similar proof.
\end{proof}

\section{Bounded geometry}\label{s:bounded geometry}

Recall the following terminology for a Riemannian manifold $N$. The {\em injectivity radius\/} of $N$ at a point $p$ is the supremum of all $r>0$ such that the exponential map is defined and is a diffeomorphism on the open $r$-ball of center $0$ in $T_pN$. The {\em injectivity radius\/} of $N$ is the infimum of the injectivity radii at all of its points.

The {\em leafwise injectivity radius\/} of $\FF$ is the injectivity radius of the disjoint union of the leaves.  On the other hand, given a defining cocycle $\{U_a,\pi_a,h_{ab}\}$ of $\FF$ and a covering $\{Q_a\}$ of $M$ consisting of compact sets $Q_a\subset U_a$, the {\em transverse injectivity radius\/} of $\FF$, with respect to $\{U_a,\pi_a,h_{ab}\}$ and $\{Q_a\}$, is the infimum of the injectivity radii of $\Sigma=\bigsqcup_a\Sigma_a$ at the points of $\bigsqcup_a\pi_a(Q_a)$, with respect to the induced metric. The condition of having a positive transverse injectivity radius is independent of $\{U_a,\pi_a,h_{ab}\}$ and $\{Q_a\}$ satisfying the stated conditions.

\begin{defn}\label{d:bounded geometry}
	It is said that $\FF$ is of {\em bounded geometry\/} if it has positive leafwise and transverse injectivity radii, and the functions $|\nabla^mR|$, $|\nabla^m\sT|$ and $|\nabla^m\sA|$ are uniformly bounded on $M$ for every $m\in\N$.
\end{defn}

\begin{rem}
	If $\FF$ is of bounded geometry, then the disjoint union of its leaves is of bounded geometry by~\eqref{rR(V,W)} and~\eqref{rnabla_V W}.
\end{rem}

\begin{exs}
	\begin{itemize}
		
		\item[(i)] Trivially, Riemannian foliations on compact manifolds are of bounded geometry for any bundle-like metric. 
		
		\item[(ii)] Let $\FF$ be a Riemannian foliation of bounded geometry on a connected manifold $M$ with a bundle-like metric $g$. Then the lift of $\FF$ to any connected covering $\widetilde M$ of $M$, equipped with the lift $\tilde g$ of $g$, is of bounded geometry.
		
		\item[(iii)] Let $H$ be a connected Lie group equipped with a left-invariant Riemannian metric $\tilde g$, and let $L\subset H$ be a connected Lie subgroup. Then the right translates of $L$ define a Riemannian foliation $\widetilde\FF$ on $H$, which is of bounded geometry with the bundle-like metric $\tilde g$. If moreover $L\vartriangleleft H$ and $\Gamma\subset H$ is a discrete subgroup, then $\widetilde\FF$ and $\tilde g$ descend to $\Gamma\backslash H$, defining a Riemannian foliation $\FF$ of bounded geometry with a bundle-like metric $g$.
		
		\item[(iv)] Let $\FF$ be a codimension one foliation almost without holonomy on a compact manifold $M$; thus the leaves with non-trivial holonomy group are compact. Suppose that $\FF$ has a finite number of leaves with non-trivial holonomy groups, whose union is a closed submanifold $C\subset M$. Then the restriction of $\FF$ to $M\setminus C$ is Riemannian and has bounded geometry for some bundle-like metric \cite{AlvKordyLeichtnam:atffff}.
		
		\item[(v)] Changes of a Riemannian foliation or bundle-like metric in a compact region preserves the bounded geometry condition.
		
	\end{itemize}
\end{exs}

\begin{thm}\label{t:bounded geometry}
  With the above notation, $\FF$ is of bounded geometry if and only if there is a normal foliation chart $x_p=(x_p^1,\dots,x_p^n):U_p\to B'\times B''$ at each $p\in M$ such that the balls $B'$ and $B''$ are independent of $p$, and the corresponding metric coefficients $g^p_{ij}$ and $g_p^{ij}$, as family of smooth functions on $B'\times B''$ parametrized by $i$, $j$ and $p$, lie in a bounded subset of the Fr\'echet space $C^\infty(B'\times B'')$.
\end{thm}

\begin{proof}
	To prove the ``if'' part, take normal foliation coordinates $x_p:U_p\to B'\times B''$ at each $p\in M$ satisfying the conditions of the statement. Then the injectivity radius of the leaf $L_p$ at $p$ is obviously greater or equal than the Euclidean radius of $B''$ since each restriction $x_p'':U_p\cap L_p\to B''$ is a system of normal coordinates of $L_p$ at $p$. Thus $\FF$ has a positive leafwise injectivity radius.  
	
	After shrinking the open covering $\{U_p\}$ if necessary, we can assume that there is a defining cocycle of the form $\{U_p,x_p',h_{pq}\}$, with $p,q\in M$ and $x_p':U_p\to B'$. Then, with the metric induced by each $\pi_p'$, the injectivity radius of $B'$ at the point $\pi_p(p)=0$ equals the Euclidean radius of $B'$. So the condition to have positive transverse injectivity radius is satisfied with $\{U_p,x_p',h_{pq}\}$ and the covering of $M$ by the sets $Q_p=\{p\}$. On the other hand, by using~\eqref{sV, sH} and the expression of the Christoffel symbols of the Levi-Civita connection in terms of the metric coefficients \cite[Chapter~2, page~56]{doCarmo1992}, it follows that the coefficients of $R$, $\sT$ and $\sA$ with respect to the foliation charts $x_p$ have polynomial expressions in terms of the coefficients $g^p_{ij}$ and $g_p^{ij}$ and their partial derivatives, obtaining easily from the hypothesis that $|\nabla^mR|$, $|\nabla^m\sT|$ and $|\nabla^m\sA|$ are uniformly bounded on $M$ for every $m\in\N$.
	
	Now assume that $\FF$ is of bounded geometry to prove the ``only if'' part. Since $\FF$ has positive leafwise and transverse injectivity radii, it easily follows that there are normal foliation coordinates $x_p:U_p\to B'\times B''$ at each $p\in M$ with $B'$ and $B''$ independent of $p$. We will use the following terminology: for functions of the metric coefficients and their partial derivatives with respect to these charts, saying that they are uniformly bounded on $B\times B'$ will mean that they are uniformly bounded as functions of $x\in B'\times B''$ and $p\in M$. This kind of terminology will be also used on $B'\times\{0\}$, as well as for the condition of being uniformly bounded away from zero. Let us drop the index $p$ from the notation, like in Section~\ref{s: coefficients}.  
	
	By~\eqref{g_ij}, it is enough to prove that the functions $\partial_Ia^i_j$ and $\partial_Ib^i_j$ are uniformly bounded on $B'\times B''$, independently of $p$, for each multi-index $I$. 
	
	According to \cite{Eichhorn1991,Schick1996,Schick2001}, the functions $\partial'_Ig_{ij}$ and $\partial'_Ig^{ij}$ are uniformly bounded on $B'\times\{0\}$, independently of $p$, for all $I$ and $i,j\le n'$. Moreover $g_{ij}=g^{ij}=\delta^i_j$ on $B'\times\{0\}$ if $\max\{i,j\}>n'$. Then, by  Propositions~\ref{partial'_i(x',tx'') is Jacobi} and~\ref{p:|X|^2+|Y|^2}, and Corollary~\ref{c:|sH X| is constant}, the functions $g_{ii}$ are uniformly bounded and uniformly bounded away from zero, independently of $p$, on $B'\times B''$. Therefore all functions $g_{ij}$ and $g^{ij}$ are uniformly bounded, independently of $p$, on $B'\times B''$. This also means that the functions $a^i_j$ and $b^i_j$ are uniformly bounded, independently of $p$. Then, according to \cite[Lemma~2.17]{Schick2001}, it is indeed enough to prove that the functions $\partial'_Ia^i_j$ are uniformly bounded on $B'\times B''$ for each $I$, independently of $p$. To establish this property, we show by induction on $r\in\N$ that the partial derivatives up to order $r$ of $a^*_*$, $\rR^{\underline{*}}_{\underline{*}\underline{*}\underline{*}}$, $\rT^{\underline{*}}_{\underline{*}\underline{*}}$, $\rR^*_{***}$ and $\rT^*_{**}$, as well as the partial derivatives up to order $r-1$ of $\rGamma^{\underline{*}}_{\underline{*}*}$ and $\rGamma^*_{**}$, are uniformly bounded on $B'\times B''$, independently of $p$.

In the case $r=0$, we have already indicated that $a^*_*$ is uniformly bounded, independently of $p$. Also $\rT^{\underline{*}}_{\underline{*}\underline{*}}$ and $\rR^{\underline{*}}_{\underline{*}\underline{*}\underline{*}}$ are uniformly bounded, independently of $p$, because the frame $s_1,\dots,s_n$ is orthonormal, and $|\rT|$ and $|\rR|$ are uniformly bounded on $M$ by~\eqref{rT(V,W)}--\eqref{rR(X,V)}. Hence $\rR^*_{***}$ and $\rT^*_{**}$ are uniformly bounded, independently of $p$, by~\eqref{rT^underline i_kl}--\eqref{rR^i_jkl;m_1 dots m_r}. In this case there is nothing to prove about $\rGamma^{\underline{*}}_{\underline{*}*}$ and $\rGamma^*_{**}$.
	
	Now suppose that the stated property holds for some natural $r$, and let us show it for $r+1$. Since $|\nabla^mR|$, $|\nabla^m\sT|$ and $|\nabla^m\sA|$ are uniformly bounded on $M$ for all $m$, it follows by~\eqref{nabla-rnabla}--\eqref{rR(X,V)} that $|\rnabla^m\rT|$ and $|\rnabla^m\rR|$ are uniformly bounded on $M$ for all $m$. So all functions $\rT^{\underline{*}}_{\underline{*}\underline{*};\underline{*}\dots\underline{*}}$ and $\rR^{\underline{*}}_{\underline{*}\underline{*}\underline{*};\underline{*}\dots\underline{*}}$ are uniformly bounded, independently of $p$, because $s_1,\dots,s_n$ is orthonormal. By~\eqref{rT^underline i_kl}--\eqref{rR^i_jkl;m_1 dots m_r}, it follows that $\rT^{\underline{*}}_{**;*\dots*}$, $\rR^{\underline{*}}_{\underline{*}**;*\dots*}$, $\rT^*_{**;*\dots*}$ and $\rR^*_{***;*\dots*}$ are also uniformly bounded, independently of $p$. Using~\eqref{rR^i_jkl;I}, Lemma~\ref{l:rR^underline i_underline j kl;I} and the induction hypothesis, we get that the partial derivatives up to order $r$ of $\rT^{\underline{*}}_{\underline{*}\underline{*}}$, $\rR^{\underline{*}}_{\underline{*}\underline{*}\underline{*}}$, $\rT^*_{**}$ and $\rR^*_{***}$ are uniformly bounded, independently of $p$. Hence the partial derivatives up to order $r$ of $\rGamma^{\underline{*}}_{\underline{*}*}$ and $\rGamma^*_{**}$ are uniformly bounded, independently of $p$, by Corollaries~\ref{c:equations relating Cristoffel symbols and curvature, underlined},~\ref{c:equations relating the derivatives of Cristoffel symbols and curvature, underlined},~\ref{c:equations relating Cristoffel symbols and curvature, non-underlined} and~\ref{c:equations relating the derivatives of Cristoffel symbols and curvature, non-underlined}. Then, using again~\eqref{rR^i_jkl;I} and Lemma~\ref{l:rR^underline i_underline j kl;I}, we get uniform bounds, independent of $p$, also for the partial derivatives of order $r+1$ of $\rT^{\underline{*}}_{\underline{*}\underline{*}}$, $\rR^{\underline{*}}_{\underline{*}\underline{*}\underline{*}}$, $\rT^*_{**}$ and $\rR^*_{***}$. Therefore, by Corollary~\ref{c:(partial_alpha xi_ij)'(t,x)} and the induction hypothesis, there are some $C_1,C_2\ge0$, independent of $p$, such that
		\begin{align*}
			|(\partial'_I\xi_{ij})'(t,x_0)|
			&\le C_1\max_{\alpha,\beta,\ 0\le\tau\le t}|\partial'_I\xi_{\alpha\beta}(\tau,x_0)|+C_2\;,\\
			|(\partial_I\hat\xi_{ij})'(t,x)|
			&\le C_1\max_{\alpha,\beta,\ 0\le\tau\le t}|\partial_I\hat\xi_{\alpha\beta}(\tau,x)|+C_2\;,
		\end{align*} 
	for all $x_0\in B'\times\{0\}$, $x\in B'\times B''$, $I$, $i$ and $j$ with $|I|=r+1$. Using these inequalities and the argument of the end of the proof of \cite[Theorem~2.5(a1)]{Schick2001}, we get that $\partial_Ia^i_j(x)=\partial_I\hat\xi_{ij}(1,x)$ is uniformly bounded, independently of $p$. 
\end{proof}

\begin{rem}
	By Theorem~\ref{t:bounded geometry} and~\eqref{theta_XV}, if $\FF$ is of bounded geometry in our sense, then it satisfies Sanguiao's definition of bounded geometry  \cite[Definition~2.7]{Sanguiao2008}.
\end{rem}

Suppose from now on that $\FF$ is of bounded geometry, and consider the foliation charts $x_p:U_p\to B'\times B''$ given by Theorem~\ref{t:bounded geometry}. 

\begin{prop}\label{p: r}
	There is some $r>0$ such that the injectivity radius of $M$ is $\ge r$ and $B(p,r)\subset U_p$ for all $p\in M$; in particular, $M$ is of bounded geometry. 
\end{prop}

\begin{proof}
	If $B'\times B''$ were a closed manifold, then the result would follow immediately from the well-known  continuity of the injectivity radius on closed manifolds \cite{Ehrlich1974,Sakai1983}. Since $B'\times B''$ is not compact, we embed it into a closed manifold to apply that result.
	
	Let $\phi:B'\times B''\to N$ be a smooth open embedding into any fixed closed manifold $N$, necessarily of dimension $n$. Let $q=\phi(0)$ and $V=\phi(B'\times B'')$, and let $W$ be an open neighborhood of $q$ in $N$ with $\overline{W}\subset V$. Let $\{\lambda,\mu\}$ be a partition of unity of $N$ subordinated to the open covering $\{V,N\setminus\overline{W}\}$. For each $p\in M$, let $g_p=g|_{U_p}$, which is considered as a metric in $V$ via $\phi x_p$. Fix any metric $g'$ on $N$. According to Theorem~\ref{t:bounded geometry}, the metrics $g'_p=\lambda g_p+\mu g'$ form a bounded subset in the Fr\'echet space of Riemannian metrics on $N$ with the $C^\infty$ topology. By the continuity of the injectivity radius on closed manifolds  \cite{Ehrlich1974,Sakai1983}, it follows that there is some $r>0$ such that the injectivity radius of $(N,g'_p)$ at $q$ is $\ge r$ for all $p\in M$. Moreover we can assume that the open $g'_p$-ball of center $q$ and radius $r$ is contained in $W$ for all $p$. Then the result follows.
\end{proof}

Let us use the notation $\partial_{p,i}$, $\partial_{p,I}$ and $\rGamma^i_{p,jk}$ to indicate that $\partial_i$, $\partial_I$ and $\rGamma^i_{jk}$ are defined by $x_p$.

\begin{prop}\label{p: changes of normal foliation coordinates}
	For each $m\in\N$, there is some $C_m>0$ such that $|\partial_I(x_qx_p^{-1})|\le C_m$ on $x_p(U_p\cap U_q)$ for all $p,q\in M$ and multi-indices $I$ with $|I|\le m$. 
\end{prop}

\begin{proof}
	We adapt the proof of \cite[Theorem~A.22]{Schick1996}. 
	
	Consider first the case where $q\in U_p$. By~\eqref{sV, sH} and the expression of the Christoffel symbols of $\nabla$ in terms of the metric coefficients \cite[Chapter~2, page~56]{doCarmo1992}, the functions $\rGamma_{p,jk}^i$ are given by expressions involving the functions $g_p^{ij}$, $g^p_{ij}$ and their partial derivatives. Moreover the $\rnabla$-geodesics in $U_p$ are given by a second order ODE, and the $\rnabla$-parallel transport along these $\rnabla$-geodesics is given by a first order ODE, both of them involving the functions $\rGamma^i_{p,jk}$. Then, by Theorem~\ref{t:bounded geometry} and \cite[Corollary~A.25]{Schick1996}, there is some $C_{0,m}>0$, independent of $p$ and $q$, such that $|\partial_{p,I}(x_qx_p^{-1})|\le C_{0,m}$ around $x_p(q)$ for all multi-index $I$ with $|I|\le m$.
	
	In the arbitrary case, given any $q_0\in U_p\cap U_q$, we apply the above case to $x_qx_{q_0}^{-1}$ and $x_{q_0}x_p^{-1}$, and use the chain rule with the expression $x_qx_p^{-1}=(x_qx_{q_0}^{-1})(x_{q_0}x_p^{-1})$ around $x_p(q_0)$.	
\end{proof}

For $m\in\N$, let $C_b^m(M)$ be the set of $C^m$ functions $f$ on $M$ such that there is some $C_m\ge0$ so that $|\nabla^kf|\le C_m$ if $k\le m$. Let $\| f\|_{C_b^m}$ be the smallest constant $C_m$ satisfying this condition with each $f\in C_b^m(M)$. This defines a norm $\|\ \|_{C_b^m}$ on $C_b^m(M)$, obtaining a Banach space. Observe that $C^{m+1}_b(M)\subset C^m_b(M)$, continuously, for each $m$, and let $C^\infty_b(M):=\bigcap_mC_b^m(M)$ with the corresponding Fr\'echet topology. By Theorem~\ref{t:bounded geometry}, the norm $\|\ \|_{C_b^m}$ is equivalent to the norm $\|\ \|'_{C_b^m}$, where $\| f\|'_{C_b^m}$ is the smallest constant $C'_m\ge0$ such that $|\partial_{p,I}f|\le C'_m$ on $U_p$ for all $p\in M$ and multi-indices $I$ with $|I|\le m$. By Proposition~\ref{p: changes of normal foliation coordinates},  another equivalent norm is defined in the same way by taking only a subset of points $p$ so that the corresponding sets $U_p$ cover $M$.

Let $r'_0$ and $r''_0$ denote the Euclidean radii of the balls $B'$ and $B''$. For $0<r'\le r'_0$ and $0<r''\le r''_0$, let $B'_{r'}$ and $B''_{r''}$ denote the balls in $\R^{n'}$ and $\R^{n''}$ centered at the origin with radii $r'$ and $r''$, respectively, and set $U_{p,r',r''}=x_p^{-1}(B'_{r'}\times B''_{r''})$. Observe that
	\begin{equation}\label{U_p,r',r'' subset B(p,r'+r'')}
		U_{p,r',r''}\subset B(p,r'+r'')\;.
	\end{equation}

\begin{prop}\label{p: partition of unity}
	Let $r',r''>0$ with $2r'\le r'_0$ and $2r''\le r''_0$. Then the following properties hold:
		\begin{itemize}
	
			\item[(i)] There is a collection of points $p_i$ in $M$, and there is some $N\in\N$ such that the sets $U_{p_i,r',r''}$ cover $M$, and each intersection of $N+1$ sets $U_{p_i,2r',2r''}$ is empty.
	
			\item[(ii)] There is a partition of unity $\{\phi_i\}$ subordinated to the open covering $\{U_{p_i,2r',2r''}\}$, which is bounded in the Fr\'echet space $C_b^\infty(M)$.
		
		\end{itemize}
\end{prop}

\begin{proof}
	By Proposition~\ref{p: r}, there is some $r_1>0$ such that $B(p,r_1)\subset U_{p,r',r''}$ for all $p\in M$. Take a family of points $p_i$ in $M$ so that $\{B(p_i,r_1/2)\}$ is a maximal collection of disjoint balls of radius $r_1/2$, whose existence is given by Zorn's lemma. Then $\{B(p_i,r_1)\}$ covers $M$, and therefore $\{U_{p_i,r',r''}\}$ also covers $M$. Let $r_2=2r'+2r''$. By Proposition~\ref{p: r} and \cite[Lemma~3.20]{Schick1996}, there is some $N\in\N$ such that
		\begin{equation}\label{vol}
			\vol B(p,r_1+r_2)\le N\vol B(p,r_1/2)
		\end{equation}
	for all $p\in M$. If some point $p$ belongs to $N+1$ sets $U_{p_i,2r',2r''}$, then it belongs to $N+1$ balls $B(p_i,r_2)$ by~\eqref{U_p,r',r'' subset B(p,r'+r'')}. It follows that $B(p,r_1+r_2)$ contains $N+1$ of the disjoint balls $B(p_i,r_1/2)$, which contradicts~\eqref{vol}. This shows~(i).
	
	To prove~(ii), take $C^\infty$ functions, $\rho'$ on $B'$ and $\rho''$ on $B''$, such that $0\le\rho'\le1$, $0\le\rho''\le1$, $\supp\rho'\subset B'_{2r'}$, $\supp\rho''\subset B''_{2r''}$, $\rho'=1$ on $B_{r'}$, and $\rho''=1$ on $B_{r''}$. Let $\psi_i$ be the $C^\infty$ function on $M$ supported in $U_{p_i}$ such that $\psi_ix_{p_i}^{-1}(x',x'')=\rho'(x')\rho''(x'')$ for all $(x',x'')\in B'\times B''$. Clearly, $\{\psi_i\}$ is a bounded subset of $C^\infty_b(M)$, and, by~(i), $\psi=\sum_i\psi_i$ is a positive function in $C^\infty_b(M)$, as well as $1/\psi$. Hence~(ii) follows with $\phi_i=\psi_i/\psi$.
\end{proof}

\bibliographystyle{amsplain}


\providecommand{\bysame}{\leavevmode\hbox to3em{\hrulefill}\thinspace}
\providecommand{\MR}{\relax\ifhmode\unskip\space\fi MR }
\providecommand{\MRhref}[2]{%
  \href{http://www.ams.org/mathscinet-getitem?mr=#1}{#2}
}
\providecommand{\href}[2]{#2}

\end{document}